\documentclass[11pt]{amsart}
\usepackage{amsopn}
\usepackage{amssymb, amscd}
\usepackage{paracol}
\usepackage{graphicx, graphics, epsfig}
\usepackage{faktor}
\usepackage{enumerate}
\usepackage{hyperref}
\globalcounter{figure}

\graphicspath{ {imagenes/} }
\usepackage{tikz}
\usepackage{pgfplots}

\newcommand{\nc}{\newcommand}

\nc{\fg}{\mathfrak{f} } \nc{\vg}{\mathfrak{v} } \nc{\wg}{\mathfrak{w} }
\nc{\zg}{\mathfrak{z} } \nc{\ngo}{\mathfrak{n} } \nc{\kg}{\mathfrak{k} }
\nc{\mg}{\mathfrak{m} } \nc{\bg}{\mathfrak{b} } \nc{\ggo}{\mathfrak{g} }
\nc{\ggob}{\overline{\mathfrak{g}} } \nc{\sog}{\mathfrak{so} }
\nc{\sug}{\mathfrak{su} } \nc{\spg}{\mathfrak{sp} } \nc{\slg}{\mathfrak{sl} }
\nc{\glg}{\mathfrak{gl} } \nc{\cg}{\mathfrak{c} } \nc{\rg}{\mathfrak{r} }
\nc{\hg}{\mathfrak{h} } \nc{\tgo}{\mathfrak{t} } \nc{\ug}{\mathfrak{u} }
\nc{\dg}{\mathfrak{d} } \nc{\ag}{\mathfrak{a} } \nc{\pg}{\mathfrak{p} }
\nc{\sg}{\mathfrak{s} } \nc{\affg}{\mathfrak{aff} } \nc{\qg}{\mathfrak{q} }

\nc{\pca}{\mathcal{P}} \nc{\nca}{\mathcal{N}} \nc{\lca}{\mathcal{L}}
\nc{\oca}{\mathcal{O}} \nc{\mca}{\mathcal{M}} \nc{\tca}{\mathcal{T}}
\nc{\aca}{\mathcal{A}} \nc{\cca}{\mathcal{C}} \nc{\gca}{\mathcal{G}}
\nc{\sca}{\mathcal{S}} \nc{\hca}{\mathcal{H}} \nc{\bca}{\mathcal{B}}
\nc{\dca}{\mathcal{D}} \nc{\val}{\operatorname{val}}

\nc{\vp}{\varphi} \nc{\ddt}{\frac{d}{dt}} \nc{\dds}{\frac{d}{ds}}
\nc{\dpar}{\frac{\partial}{\partial t}} \nc{\im}{\mathrm{i}}

\nc{\SO}{\mathrm{SO}} \nc{\Spe}{\mathrm{Sp}} \nc{\Sl}{\mathrm{SL}}
\nc{\SU}{\mathrm{SU}} \nc{\Or}{\mathrm{O}} \nc{\U}{\mathrm{U}} \nc{\Gl}{\mathrm{GL}}
\nc{\Se}{\mathrm{S}} \nc{\Cl}{\mathrm{Cl}} \nc{\Spein}{\mathrm{Spin}}
\nc{\Pin}{\mathrm{Pin}} \nc{\G}{\mathrm{GL}_n(\RR)} \nc{\g}{\mathfrak{gl}_n(\RR)}

\nc{\RR}{{\Bbb R}} \nc{\HH}{{\Bbb H}} \nc{\CC}{{\Bbb C}} \nc{\ZZ}{{\Bbb Z}}
\nc{\FF}{{\Bbb F}} \nc{\NN}{{\Bbb N}} \nc{\QQ}{{\Bbb Q}} \nc{\PP}{{\Bbb P}} \nc{\OO}{{\Bbb O}}

\nc{\vs}{\vspace{.2cm}} \nc{\vsp}{\vspace{1cm}} \nc{\ip}{\langle\cdot,\cdot\rangle}
\nc{\ipp}{(\cdot,\cdot)} \nc{\la}{\langle} \nc{\ra}{\rangle} \nc{\unm}{\tfrac{1}{2}}
\nc{\unc}{\tfrac{1}{4}} \nc{\und}{\tfrac{1}{16}} \nc{\no}{\vs\noindent}
\nc{\lam}{\Lambda^2(\RR^n)^*\otimes\RR^n} \nc{\tangz}{{\rm T}^{\rm Zar}}
\nc{\nor}{{\sf n}}  \nc{\mum}{/\!\!/} \nc{\kir}{/\!\!/\!\!/}
\nc{\Ri}{\tfrac{4\Ric_{\mu}}{||\mu||^2}} \nc{\ds}{\displaystyle}
\nc{\ben}{\begin{enumerate}} \nc{\een}{\end{enumerate}} \nc{\f}{\frac}
\nc{\lb}{[\cdot,\cdot]} \nc{\isn}{\tfrac{1}{||v||^2}}
\nc{\gkp}{(\ggo=\kg\oplus\pg,\ip)} \nc{\ukh}{(\ug=\kg\oplus\hg,\ip)}
\nc{\tgkp}{(\tilde{\ggo}=\kg\oplus\pg,\ip)}
\nc{\wt}{\widetilde}
\nc{\iop}{\mathtt{i}} \nc{\jop}{\mathtt{j}}

\nc{\alp}{\alpha^2}  \nc{\bet}{\beta^2}  \nc{\gam}{\gamma^2} \nc{\et}{\eta^2}

\nc{\Hess}{\operatorname{Hess}} \nc{\ad}{\operatorname{ad}}
\nc{\Ad}{\operatorname{Ad}} \nc{\rank}{\operatorname{rank}}
\nc{\Irr}{\operatorname{Irr}} \nc{\End}{\operatorname{End}}
\nc{\Aut}{\operatorname{Aut}} \nc{\Inn}{\operatorname{Inn}}
\nc{\Der}{\operatorname{Der}} \nc{\Ker}{\operatorname{Ker}}
\nc{\Iso}{\operatorname{Iso}} \nc{\Diff}{\operatorname{Diff}}
\nc{\Lie}{\operatorname{L}} \nc{\tr}{\operatorname{tr}} \nc{\dif}{\operatorname{d}}
\nc{\sen}{\operatorname{sen}} \nc{\modu}{\operatorname{mod}}
\nc{\CRic}{\operatorname{PP}} \nc{\Cric}{\operatorname{P}} \nc{\Ricci}{\operatorname{Ric}}
\nc{\sym}{\operatorname{sym}} \nc{\herm}{\operatorname{herm}} \nc{\symac}{\operatorname{sym^{ac}}}
\nc{\symc}{\operatorname{sym^{c}}} \nc{\scalar}{\operatorname{scal}}
\nc{\grad}{\operatorname{grad}} \nc{\ricci}{\operatorname{Rc}}
\nc{\Nor}{\operatorname{Norm}}  \nc{\ricc}{\operatorname{Rc^{c}}}
\nc{\Ricc}{\operatorname{Ric^{c}}} \nc{\ricac}{\operatorname{Rc^{ac}}}
\nc{\Ricac}{\operatorname{Ric^{ac}}} \nc{\Riem}{\operatorname{Rm}} \nc{\Sec}{\operatorname{Sec}}
\nc{\riccig}{\operatorname{ric^{\gamma}}} \nc{\Rin}{\operatorname{M}}
\nc{\Le}{\operatorname{L}} \nc{\tang}{\operatorname{T}}
\nc{\level}{\operatorname{level}} \nc{\rad}{\operatorname{r}}
\nc{\abel}{\operatorname{ab}} \nc{\CH}{\operatorname{CH}} \nc{\Cone}{{\mathcal C}} \nc{\CCone}{\operatorname{CC}} \nc{\CP}{{\mathcal P}}
\nc{\mcc}{\operatorname{mcc}} \nc{\Adj}{\operatorname{Adj}}
\nc{\Order}{\operatorname{O}}  \nc{\inj}{\operatorname{inj}} \nc{\proy}{\operatorname{pr}}
\nc{\vol}{\operatorname{vol}} \nc{\Diag}{\operatorname{Dg}} \nc{\Diagg}{\operatorname{Diag}}
\nc{\Spec}{\operatorname{Spec}} \nc{\Ima}{\operatorname{Im}} \nc{\Rea}{\operatorname{Re}}
\nc{\spann}{\operatorname{span}} \nc{\Aff}{\operatorname{Aff}}
\nc{\mm}{\operatorname{m}}

\newtheorem{theorem}{Theorem}[section]
\newtheorem{proposition}[theorem]{Proposition}
\newtheorem{corollary}[theorem]{Corollary}
\newtheorem{lemma}[theorem]{Lemma}

\theoremstyle{definition}
\newtheorem{definition}[theorem]{Definition}

\theoremstyle{remark}
\newtheorem{remark}[theorem]{Remark}

\title{On Ricci negative derivations}

\author{Mar\'{\i}a Valeria Guti\'{e}rrez}

\address{Universidad Nacional de C\'ordoba, FAMAF and CIEM, 5000 C\'ordoba, Argentina}
\email{valeria.gutierrez@mi.unc.edu.ar}

\thanks{Partially supported by a CONICET doctoral fellowship and a Consejo Interuniversitario Nacional fellowship}

\begin{document}

 \maketitle

\begin{abstract}
Given a nilpotent Lie algebra, we study the space of all diagonalizable derivations such that the corresponding one-dimensional solvable extension admits a left-invariant metric with negative Ricci curvature.  It has been conjectured by Lauret-Will that such a space coincides with an open and convex subset of derivations defined in terms of the moment map for the variety of nilpotent Lie algebras.  We prove the validity of the conjecture in dimension $\leq 5$, as well as for Heisenberg and standard filiform Lie algebras.
\end{abstract}

\tableofcontents

\section{Introduction}
 Even though in the general case there are no topological obstructions on a differentiable manifold $M$ to the existence of a complete Riemannian metric with $\Ricci<0$, in the homogeneous case, this is the only curvature behavior which is still far from being understood. Given a Lie group $G$, a nice relationship between any prescribed curvature behavior of left-invariant metrics, the topology of $G$ and also the structure of its Lie algebra $\ggo$ is expected. When we consider a Lie group all the metrics are assumed to be left invariant.

  In \cite{DttLtMtl}, Dotti, Leite and Miatello proved that if a unimodular Lie group $G$ admits a $\Ricci<0$ metric, then $G$ is non-compact and semisimple. They also showed that most of non-compact, simple Lie groups indeed have one, with some low dimensional exceptions. On the other hand, it was proved by Jablonski and Petersen \cite{JblPtr} that a semisimple Lie group admitting a metric with $\Ricci<0$ can not have compact factors.

  In 2016, Will constructed unexpected examples of Lie groups admitting a $\Ricci<0$ metric which are neither semisimple nor solvable \cite{Wll1, Wll2}. Furthermore, a general construction in  \cite{Wll2} shows that any non-compact, semisimple Lie group admitting a $\Ricci<0$ metric can be the Levi factor of a non-semisimple Lie group with a $\Ricci<0$ metric, and in \cite{LrtWll} it is proved that any compact semisimple Lie group can be the Levi factor of a Lie group having a $\Ricci<0$ metric. All this shows that an algebraic characterization of Lie groups admitting this kind of metrics is out of reach at the moment.

 In the solvable case, Nikolayevsky and Nikonorov obtained a sufficient condition on a solvable Lie group $S$, with Lie algebra $\sg$, for the existence of a metric on $S$  with $\Ricci<0$: there exists $Y\in\sg$ such that all the eigenvalues of the restriction of $\ad{Y}$ to the nilradical of $\sg$ have positive real part. They also prove a necessary condition, which says that if $S$ admits a metric of negative Ricci curvature, then there exists $Y\in\sg$ such that $\tr{\ad{Y}}>0$ and all the eigenvalues of the restriction of the operator $\ad{Y}$ to the center of the nilradical have a positive real part.

It has been shown in \cite{RN} that the class of those nilpotent Lie algebras that can be the nilradical of some solvable Lie algebra admitting a $\Ricci<0$ metric is really far from being understood. We consider instead the following question.
\begin{description}\label{P2}
\item[Question 1]
Given a nilpotent Lie algebra $\ngo$, which are the solvable Lie algebras with nilradical $\ngo$ admitting a $\Ricci<0$ metric?
\end{description}

The following conjecture was proposed in \cite{NklNkn,Nkl}, see also \cite{RNder}.
\begin{quote}
  For each nilpotent Lie algebra $\ngo$, there is an open and convex cone $\Cone$ in the maximal torus of derivations of $\ngo$ such that a solvable Lie algebra $\sg$ with nilradical $\ngo$ admits a $\Ricci<0$ metric if and only if there exists $Y\in\sg$ such that $\ad{Y}|_\ngo^\RR\in \Cone$ (up to automorphism conjugation).
\end{quote}

In the context of solvable Lie groups, strongly Ricci negative derivations of a nilpotent Lie algebra $\ngo$ were defined in \cite{RNder} in order to study this problem by considering one-dimensional solvable extensions of $\ngo$.
\begin{definition}[see Definition 2.1]
A derivation $D$ of a nilpotent Lie algebra $\ngo$ with $\tr{D}>0$ is said to be {\it strongly Ricci negative} if the one-dimensional solvable extension Lie algebra $\sg_D=\RR f\oplus\ngo$ ($\ad{f}|_{\ngo}=D$) admits an inner product of negative Ricci curvature such that $D^t=D$ and $f\perp\ngo$.
\end{definition}
If we denote by $\tgo(\ngo)$ the maximal torus of diagonalizable (over $\RR$) derivations of $\ngo$, we call $\tgo(\ngo)_{srn}$ the set of all derivations in $\tgo(\ngo)$ which are strongly Ricci negative. In \cite{RNder} it is also defined the following open and convex cone $\Cone(\ngo)$ consisting of this kind of derivations:
\begin{equation*}
\Cone(\ngo):=\left(\RR_{>0}\, \mm\left(\overline{G_{\tgo(\ngo)}\cdot\lb}\right)\cap\Diag(\ngo) + \Diag(\ngo)_{>0}\right)\cap\tgo(\ngo)_{\tr>0},
\end{equation*}
where $G_{\tgo(\ngo)}$ denotes the connected component of the centralizer of $\tgo(\ngo)$ in $\Gl(\ngo)$, which acts on $V:=\Lambda^2\ngo^*\otimes\ngo$ by $g\cdot\mu:=g\mu(g^{-1}\cdot,g^{-1}\cdot)$. We denote by $\mm$ the moment map for the $\Gl(\ngo)$-representation on $V$ defined by the derivative of this action. The vector spaces $\Diag(\ngo)$ and $\Diag(\ngo)_{>0}$ are those of diagonal operators (in terms of a fixed basis of $\ngo$) and positive diagonal operators of $\ngo$, respectively, and $\tgo(\ngo)_{\tr>0}:=\{D\in \tgo(\ngo) : \tr D >0\}$.

The following inclusions are proved in \cite{RNder}:
$$\tgo(\ngo)_{srn}\cap\tgo(\ngo)_{gen}\subset\Cone(\ngo)\subset\tgo(\ngo)_{srn},$$
where $\tgo(\ngo)_{gen}$ are the generic derivations in $\tgo(\ngo)$. In view of this, given a nilpotent Lie algebra $\ngo$, it is natural to expect the validity of the following
\begin{description}
  \item[Conjecture 1] $\Cone(\ngo)=\tgo(\ngo)_{srn}.$
\end{description}
We refer to Section \ref{preli} for a more detailed treatment. Our aim in this paper is to prove the validity of the above conjecture among three families of nilpotent Lie algebras.

In Section \ref{conosdim5}, we focus on nilpotent Lie algebras of dimension 5 and we obtain a complete description of the cone $\Cone(\ngo)$ in each case. After that, we prove our main result in Section \ref{theodim5}.

\begin{theorem}[see Theorem 4.4]
 For any nilpotent Lie algebra $\ngo$ of dimension $5$, we have that
    $\tgo(\ngo)_{srn}=\Cone(\ngo)$.
  \end{theorem}

We also consider two important classes of nonabelian nilpotent Lie algebras: the Heisenberg Lie algebras and the
standard ﬁliform Lie algebras.

 \begin{proposition}[see Proposition 5.1]
 Let $\mathfrak{h}_{2n+1}$ be the Heisenberg Lie algebra with basis $\{e_1,e_2,\ldots, e_{2n+1}\}$ and Lie brackets:
 $$ \mu(e_1,e_2)=e_{2n+1}, \quad \mu(e_3,e_4)=e_{2n+1}, \quad \ldots, \quad \mu(e_{2n-1},e_{2n})=e_{2n+1}.$$
 For $ D:= \Diag(d_1,d_{n+1}-d_1,d_2,d_{n+1}-d_2,\ldots,d_n,d_{n+1}-d_n,d_{n+1})\in \tgo(\mathfrak{h}_{2n+1}),$ the $3^n$ equations which define the cone $\Cone(\mathfrak{h}_{2n+1})$ are given by:
 \begin{equation}
 (l+1)d_{n+1}\pm d_{i_1} \pm d_{i_2} \pm \ldots \pm d_{i_k} >0, \quad 0\leq k \leq n, \quad 0\leq l \leq k.
 \end{equation}
 \end{proposition}
 This paves the way to prove Conjecture 1 for Heisenberg Lie algebras by using results of \cite{NklNkn}.
 \begin{theorem}[see Theorem 5.3]
  $\Cone(\mathfrak{h}_{2n+1})=\tgo(\mathfrak{h}_{2n+1})_{srn}, \text{ for all } n\in\NN.$
\end{theorem}

Finally, we study the standard filiform Lie algebra $L_n$.
 \begin{proposition}[see Proposition 6.1]
  Let $L_n$ be the filiform Lie algebra defined by the basis $\{e_1,e_2,\ldots, e_{n}\}$ and the Lie brackets:
  $$ \mu(e_1,e_2)=e_{3}, \quad \mu(e_1,e_3)=e_{4}, \quad \dots, \quad \mu(e_{1},e_{n-1})=e_{n}.$$ For $ D:= \Diag(d_1,d_2,d_1+d_2,2d_1+d_2,\ldots,(n-2)d_1+d_2) \in \tgo(L_n)$, the cone $\Cone(L_n)$ is defined by the following equations,
  $$(n-2)d_{1}+d_2 >0 \qquad \text{and} \qquad  \frac{(n-1)(n-2)}{2}d_1+(n-1)d_2 >0. $$
\end{proposition}

By using \cite[Theorem 4]{NklNkn}, we also prove Conjecture 1 in this case.
\begin{corollary}[see Corollary 6.2]
  $\Cone(L_n)=\tgo(L_n)_{srn}$, for all $n>3$.
\end{corollary}

\section{Preliminaries} \label{preli}

   \subsection{The representation $\lam$ and the moment map}\label{mm-sec}
   We consider the space of all skew-symmetric algebras of dimension $n$, which is parameterized by the vector space
$$
V:=\lam.
$$
There is a natural linear action of $\G$ on $V$ given by $g\cdot\mu:=g\mu(g^{-1}\cdot,g^{-1}\cdot)$, for all $g\in\G$, $\mu\in V$, whose derivative defines the following $\g$-representation on $V$,
$$
E\cdot\mu=E\mu(\cdot,\cdot)-\mu(E\cdot,\cdot)-\mu(\cdot,E\cdot),\qquad E\in\g,\quad\mu\in V.
$$
 Let $\tgo^n$ denote the set of all diagonal $n\times n$ matrices.  If $\{ e^1,...,e^n\}$ is the basis of $(\RR^n)^*$ dual to the canonical basis $\{ e_1,...,e_n\}$, then
$$
\{ \mu_{ijk}:=(e^i\wedge e^j)\otimes e_k : 1\leq i<j\leq n, \; 1\leq k\leq n\}
$$
is a basis of $V$ of weight vectors for the above representation.  Note that $\mu_{ijk}$ is actually the bilinear form on $\RR^n$ defined by
$\mu_{ijk}(e_i,e_j)=-\mu_{ijk}(e_j,e_i)=e_k$ and zero otherwise.  The corresponding weights of this basis are given by
$$
F_{ij}^k:=E_{kk}-E_{ii}-E_{jj}\in\tgo^n, \qquad i<j,
$$
where $E_{rs}$ denotes as usual the matrix whose only nonzero coefficient is $1$ at entry $rs$.  The structural constants $c(\mu)_{ij}^k$ of an
algebra $\mu\in V$ are then given by
$$
\mu(e_i,e_j)=\sum_{k}c(\mu)_{ij}^k\, e_k,
\qquad \mu=\sum_{i<j,\,k} c(\mu)_{ij}^k\, \mu_{ijk}.
$$
We endow all these vector spaces with their canonical inner products.

The moment map (or $\G$-gradient map) from real geometric invariant theory (see \cite{HnzSchStt,BhmLfn} for further information) for the above representation is the $\Or(n)$-equivariant map
$$
\mm:V\smallsetminus\{ 0\}\longrightarrow\sym(n),
$$
defined by
\begin{equation}\label{defmm}
\la \mm(\mu),E\ra=\tfrac{1}{|\mu|^2}\left\langle E\cdot\mu,\mu\right\rangle, \qquad \mu\in
V\smallsetminus\{ 0\}, \quad E\in\sym(n),
\end{equation}
or equivalently, for any $X,Y\in\RR^n$,
\begin{equation}\label{defmm2}
\la \mm(\mu)X,Y\ra = -\unm\sum\la\mu(X,e_i),e_j\ra\la\mu(Y,e_i),e_j\ra + \unc \sum\la\mu(e_i,e_j),X\ra\la\mu(e_i,e_j),Y\ra.
\end{equation}
We are using $\g=\sog(n)\oplus\sym(n)$ as a Cartan decomposition, where $\sog(n)$ and $\sym(n)$ denote the subspaces of skew-symmetric and symmetric matrices, respectively.  It is easy to check that $\mm$ is well defined on the projective space $\PP(V)$ and $\tr{\mm(\mu)}=-1$ for any $\mu\in V$.  In \cite{HnzSch, BllGhgHnz}, many nice and useful results on the convexity of the image of the moment map have been obtained, which were used in \cite{RN} to study Ricci negative solvmanifolds (see Section \ref{RN-sec}).

\subsection {Strongly Ricci negative derivations} \label{RN-sec}

Let $\ngo$ be a nilpotent Lie algebra, a maximal abelian subspace of diagonalizable (over $\RR$) derivations is called a {\it maximal torus} and denoted $\tgo(\ngo)$, it is known to be unique up to $\Aut(\ngo)$-conjugation. Each derivation $D$ of $\ngo$ defines a solvable Lie algebra
$$
\sg_D=\RR f\oplus\ngo,
$$
with Lie bracket defined as the semi-direct product such that $\ad{f}|_{\ngo}=D$.

The following strong condition for derivations was studied in \cite{RN}.

\begin{definition}\label{srnd-def}
A derivation $D$ of a nilpotent Lie algebra $\ngo$ with $\tr{D}>0$ is said to be {\it strongly Ricci negative} if the solvable Lie algebra $\sg_D$ admits an inner product of negative Ricci curvature such that $D^t=D$ and $f\perp\ngo$.  We denote by $\Der(\ngo)_{srn}$ the cone of all strongly Ricci negative derivations of $\ngo$.
\end{definition}
Since we do not know if $\Der(\ngo)_{srn}$ is open in $\Der(\ngo)$, we do not know whether the corresponding cone
\begin{equation}\label{tsrn}
\tgo(\ngo)_{srn}:=\Der(\ngo)_{srn}\cap\tgo(\ngo)
\end{equation}
is open in $\tgo(\ngo)$ either.  The cone $\tgo(\ngo)_{srn}$ was proved to be open in $\tgo(\ngo)$ and convex for Heisenberg and filiform Lie algebras endowed with the standard bases (see \cite{NklNkn,Nkl} for more details).

The next result of \cite{RN} characterizes strongly Ricci negative derivations in terms of the moment map $\mm$.  We fix a basis $\{ e_i\}$ of $\ngo$ such that $\tgo(\ngo)\subset\Diag(\ngo)$, the space of all operators of $\ngo$ whose matrix in terms of $\{ e_i\}$ is diagonal.  Note that the vector space $\ngo$ is identified with $\RR^n$ using the basis $\{ e_i\}$ and so the whole setting described in Section \ref{mm-sec} can be used.

In particular, the inner products of the definition of moment map are the canonicals, i.e. those making orthonormal the bases ${e_i\otimes e^j}$ and ${(e^i\wedge e^j)\otimes e_k}$.

Let $G_D$ denote the connected component of the centralizer subgroup of $D$ in $\Gl(\ngo)$ and let $\ggo_D$ be its Lie algebra.

\begin{theorem}\label{main}\cite[Corollary 3.4]{RN}
Let $\ngo$ be a nilpotent Lie algebra with Lie bracket $\lb$ and consider $D\in\tgo(\ngo)$ such that $\tr{D}>0$.  Then the following conditions are equivalent:
\begin{itemize}
\item[(i)] $D$ is strongly Ricci negative.
\item[ ]
\item[(ii)] $D\in\RR_{>0}\, \mm\left(G_D\cdot\lb\right)\cap\Diag(\ngo) + \Diag(\ngo)_{>0}$.
\item[ ]
\item[(iii)] $D\in\RR_{>0}\, \mm\left(\overline{G_D\cdot\lb}\right)\cap\Diag(\ngo) + \Diag(\ngo)_{>0}$.
\item[ ]
\item[(iv)] $D\in\RR_{>0}\, \mm\left(\overline{G_D\cdot\lb}\right)\cap\ag_+^D + \Diag(\ngo)_{>0}$, where $\ag_+^D\subset\Diag(\ngo)$ is any Weyl chamber of $G_D$.
\end{itemize}
\end{theorem}

Note that all the cones in the above theorem depend on $D$ and are open in $\Diag(\ngo)$ since the space $\Diag(\ngo)_{>0}$ of positive diagonal matrices is so.  On the other hand, the cone in part (iv) is in addition convex by \cite{HnzSch}.

\subsection {An open and convex cone of Ricci negative derivations}\label{opencone}

We keep the basis $\{ e_1,\dots,e_n\}$ of $\ngo$ such that $\tgo(\ngo)\subset\Diag(\ngo)$ fixed.  Let $\alpha_1,\dots,\alpha_r\in\tgo(\ngo)^*$ be the weights for $\tgo(\ngo)$ and let $\ngo=\ngo_1\oplus\dots\oplus\ngo_r$ be the decomposition of $\ngo$ in weight subspaces, that is,
$$
DX=\alpha_i(D)X, \qquad \text{for all } X\in\ngo_i, \quad D\in\tgo(\ngo).
$$
We call a $D\in\tgo(\ngo)$ {\it generic} when $\alpha_i(D)\ne\alpha_j(D)$ for any $i\ne j$.  Let $\tgo(\ngo)_{gen}$ denote the subset of all generic derivations, which is open and dense in $\tgo(\ngo)$.

If $G_{\tgo(\ngo)}$ denotes the connected component of the centralizer of $\tgo(\ngo)$ in $\Gl(\ngo)$, then $G_{\tgo(\ngo)}\subset G_D$ for any $D\in\tgo(\ngo)$ and equality holds if and only if $D\in\tgo(\ngo)_{gen}$.  In the {\it multiplicity-free} case, i.e.\ $\dim{\ngo_i}=1$ for all $i$, $G_{\tgo(\ngo)}$ is given by the torus $\Diag(\ngo)_{>0}$ with Lie algebra $\Diag(\ngo)$ and $\tgo(\ngo)$ is the only Weyl chamber.  Note that the space $\tgo(\ngo)$ is multiplicity-free if and only if there is at least one $D\in\tgo(\ngo)$ with pairwise different eigenvalues.

We consider the smallest of the open cones appearing in Theorem \ref{main}, (iii) to introduce the following cone defined in \cite{RNder}:
\begin{equation}\label{def cono}
\Cone(\ngo):=\left(\RR_{>0}\, \mm\left(\overline{G_{\tgo(\ngo)}\cdot\lb}\right)\cap\Diag(\ngo) + \Diag(\ngo)_{>0}\right)\cap\tgo(\ngo)_{\tr>0},
\end{equation}
where $\tgo(\ngo)_{\tr>0}:=\left\{ D\in\tgo(\ngo):\tr{D}>0\right\}$.

Let $O(\ngo)$ denote the orthogonal group relative to the inner product making $\{e_i\}$ an orthonormal basis. In order to have a group acting on $\Cone(\ngo)$, which may be helpful for computing it, we consider the group
\begin{equation}\label{ortweyl}
    W_{ort}(\ngo):=N_{\Aut(\ngo)\cap O(\ngo)}(\tgo(\ngo))/C_{\Aut(\ngo)\cap O(\ngo)}(\tgo(\ngo)),
   \end{equation}
where $N$ and $C$ denote normalizer and centralizer, respectively. We call it the orthogonal Weyl group.

We define the convex polytope $\Cone_t(\ngo):=\{D \in \Cone(\ngo) : \tr(D)=t\}$ which is also $W_{ort}$-invariant for all $t>0$.

\begin{proposition}\cite{RNder}
$\Cone(\ngo)$ is an open and convex cone in $\tgo(\ngo)$ such that
\begin{equation}
\tgo(\ngo)_{srn}\cap\tgo(\ngo)_{gen}\subset\Cone(\ngo)\subset\tgo(\ngo)_{srn}.
\end{equation}
\end{proposition}
This motivates the next conjecture,
\begin{description}
\item[Conjecture 1 \cite{RNder}] $\Cone(\ngo)=\tgo(\ngo)_{srn}$.
\end{description}

In the next sections, we will see that this holds for nilpotent Lie algebras of dimension $5$, as well as for Heisenberg and ﬁliform
Lie algebras.

\section{The cone $\Cone(\ngo)$ for nilpotent Lie Algebras of dim $5$}\label{conosdim5}
    In order to prove that Conjecture 1 is valid in dimension 5 (see Section \ref{opencone}), we study in this section the cones $\Cone(\ngo)$ for nilpotent Lie algebras of dimension 5, which are defined in Table \ref{clasifi}.
  \begin{table}[h]
  \renewcommand{\arraystretch}{1.5}
  \begin{center}
  \begin{tabular}{||c|c||}
      \hline\hline
      $\ngo$ & Lie bracket of $\ngo$  \\
      \hline\hline
    $\ngo_1$ &  $ \mu(e_1,e_2)=e_3, \qquad \mu(e_1,e_3)=e_4, \qquad \mu(e_1,e_4)=e_5 $ \\
    \hline
    $\ngo_2$ &  $ \mu(e_1,e_2)=e_3, \qquad \mu(e_1,e_3)=e_4, \qquad \mu(e_1,e_4)=e_5, \qquad \mu(e_2,e_3)=e_5$ \\
    \hline
    $\ngo_3$ &  $ \mu(e_1,e_2)=e_4, \qquad \mu(e_2,e_3)=e_5, \qquad \mu(e_1,e_4)=e_5 $\\
    \hline
    $\ngo_4$ & $ \mu(e_1,e_2)=e_5, \qquad \mu(e_3,e_4)=e_5$  \\
    \hline
    $\ngo_5$ &  $ \mu(e_1,e_2)=e_3, \qquad \mu(e_1,e_3)=e_4, \qquad \mu(e_2,e_3)=e_5 $ \\
    \hline
    $\ngo_6$ & $ \mu(e_1,e_2)=e_4, \qquad \mu(e_1,e_3)=e_5 $ \\
    \hline
    $\ngo_7$ &  $ \mu(e_1,e_2)=e_3$ \\
    \hline
    $\ngo_8$ & $ \mu(e_1,e_2)=e_3, \qquad \mu(e_1,e_3)=e_4 $ \\
    \hline\hline
  \end{tabular}
  \end{center}
  \caption{Nilpotent Lie Algebras of dimension 5}
  \label{clasifi}
  \end{table}

  \subsection{Case $\ngo_1$.} \label{n1}

Let $D$ be a derivation of $\ngo_1$ in the maximal torus $\tgo(\ngo_1)$, this implies that
$$D=\Diag(d_1,d_2,d_1+d_2,2d_1+d_2,3d_1+d_2), \text{ with } d_1,d_2\in\RR .$$
 $D$ is generic if and only if $d_1\neq d_2, \ d_1 \neq -d_2, \ d_2 \neq -2d_1, \text{ and } d_1,d_2 \neq 0. $
In the generic case, as $\tgo(\ngo_1)$ is multiplicity-free, $G_D=\Diag(\ngo_1)_{>0}$ and $ \overline{\Diag(\ngo_1)_{>0}\cdot \mu}$ is given by the linear subspace of V of nilpotent Lie brackets $\lambda= \lambda(\alpha, \beta,\gamma)$ deﬁned by
$$ \lambda(e_1,e_2)= \alpha e_3, \qquad \lambda(e_1,e_3)=\beta e_4, \qquad \lambda(e_1,e_4)=\gamma e_5, $$ with $\alpha,\beta,\gamma \geq 0$. The moment map of $\lambda$ is given by
 \begin{align*}
      \mm(\lambda)&=\tfrac{1}{\alpha^2 + \beta^2 + \gamma^2} \Diag(-\alp -\bet -\gam,- \alp,-\bet + \alp,-\gam + \bet,\gam) \\
       & =\tfrac{1}{\alp + \bet + \gam}(\alp F_{12}^3 + \bet F_{13}^4 + \gam F_{14}^5).
  \end{align*}

This means that $\mm(\lambda)=\mm(\overline{\Diag(\ngo_1)_{>0}\cdot \mu})=\CH (F_{12}^3, F_{13}^4, F_{14}^5).$

It follows from \eqref{def cono} that $D \in \Cone(\ngo_1) \subset \tgo(\ngo_1)$ if and only if $D=aF_{12}^3+bF_{13}^4+cF_{14}^5 + E$ and $\tr(D)>0$, where $a,b,c > 0$ and $E$ is a positive definite diagonal matrix. This can be expressed as a system of inequalities as follows:
 \begin{equation*}
   \left\{
 \begin{aligned}
  d_1+a+b+c>0,   \qquad  (1.1) \\
  d_2+a>0, \qquad  (1.2) \\
  d_1+d_2-a+b>0,  \qquad (1.3)\\
  \end{aligned}
  \right.
  \qquad \qquad
  \left\{
  \begin{aligned}
  2d_1+d_2-b+c>0,  \qquad (1.4)\\
  3d_1+d_2-c>0.  \qquad (1.5)
 \end{aligned}
      \right.
\end{equation*}
Condition $(1.5)$ implies that $3d_1+d_2 >0$ and by adding the last three inequalities we obtain $2d_1+d_2>0$.

From $(1.5)$, $(1.4)$, $(1.1)$, we have the existence of $c>0$ such that
$$ 3d_1+d_2>c>-2d_1-d_2+b \ ,  -d_1-a-b,$$
this holds if and only if, by $(1.3)$, there exists $b>0$ such that
$$5d_1+2d_2>b> -d_1-d_2+a \ , \ -4d_1-d_2-a, $$
and, consequently, there exists $a>0$ such that
$$6d_1+3d_2>a>-9d_1-3d_2 \ , \ -d_2. $$
And so, we also get $15d_1+6d_2>0$ and $6d_1+4d_2>0$.

From this, we can see that the equations which define $\Cone(\ngo_1)$ are (see Figure~\ref{cone1}):
$$ 3d_1+d_2>0, \qquad \qquad 3d_1+2d_2>0. $$

\subsection{Case $\ngo_2$}\label{n2}

Let $D=\Diag(d_1,2d_1,3d_1,4d_1,5d_1), \ d_1\in\RR$ be a derivation of $\ngo_2$ in the maximal torus $\tgo(\ngo_2)$.
If $d_1\neq 0$, then $D$ is generic and $G_D=\Diag(\ngo_2)_{>0}$, therefore $\overline{G_D\cdot\mu}=\lambda=\RR_{\geq0}\mu$

It can be seen that $\mm(\lambda)=\mm(\overline{G_D\cdot \mu})=\CH(F_{12}^3, F_{13}^4, F_{14}^5, F_{23}^5)$ and by definition \eqref{def cono} we have that $D \in \Cone(\ngo_2) \subset \tgo(\ngo_2)$ if and only if $D=aF_{12}^3+bF_{13}^4+cF_{14}^5+dF_{23}^5 + E$ where $a,b,c,d > 0$ and $E \in \Diag(\ngo_2)_{>0}$.
This means that $D \in \Cone(\ngo_2)$ if and only if $d_1>0,$ see Figure~\ref{cone2} below.

\begin{paracol}{2}
\begin{figure}[h!]
\begin{tikzpicture}[scale=0.63]
\draw[very thin,color=gray,step=1cm] (-3.2,-3.2) grid (5.5,5.5);
\foreach \x in {-1,0,1,2,3,4,5}
   \draw (\x cm,1pt) -- (\x cm,-1pt) node[anchor=north] {$\x$};
\foreach \y in {0,1,2,3,4,5}
   \draw (1pt,\y cm) -- (-1pt,\y cm) node[anchor=east] {$\y$};


\draw[->] (-3,0) -- (6,0) node[right] {$d_1$};
\draw[->] (0,-2) -- (0,5) node[above] {$d_2$};

\draw[very thick, red]  (-2,3.5) -- (12/7,-3);

\draw[dashed, very thick, blue]  (0,0) -- (-1.5,4.5);
\draw[dashed, very thick, blue]  (0,0) -- (2,-3);


\draw[orange, fill, opacity=0.2]  (-4/3,4) .. controls (4,4) .. (2,-3) -- (0,0) -- (-4/3,4);

\draw (3.8,1.5) node[color=orange, thick] {$\Cone(\ngo_1)$};
\draw (-2.3,1.5) node[color=red,very thick] {$\tr(D)>0$};
\draw (3,-2.5) node[color=blue,very thick] {$d_2=-\frac{3}{2}d_1$};
\draw (-1.8,4.7) node[color=blue,very thick] {$d_2=-3d_1$};


\end{tikzpicture}\caption{$\Cone(\ngo_1)$}
\label{cone1}
\end{figure}
\switchcolumn
\vspace{2.3cm}
\begin{figure}[h!]
\begin{tikzpicture}[scale=0.8]
\draw[very thin,color=gray,step=1cm] (-1.1,-1.1) grid (6.1,1.1);
\foreach \x in {-1,0,1,2,3,4,5}
   \draw (\x cm,1pt) -- (\x cm,-1pt) node[anchor=north] {$\x$};
\foreach \y in {}
   \draw (1pt,\y cm) -- (-1pt,\y cm) node[anchor=east] {$\y$};

\draw[->] (-0.8,0) -- (6,0) node[right] {$d_1$};

\draw[color=red, very thick] (0,0) -- (6.1,0);
\draw (3.8,1.5) node[color=red, thick] {$\Cone(\ngo_2)$};
\draw (0,0) node[color=red,very thick] {\textbf{$\left( \right.$}};

\end{tikzpicture}
\caption{$\Cone(\ngo_2)$}
\label{cone2}
\end{figure}
\end{paracol}

\subsection{Cases $\ngo_3$ and $\ngo_5$}\label{n35} For the cases $\ngo_3$ and $\ngo_5$ we are going to see the equations of the cones and the figures without detailed computes, which are very similar to the case $\ngo_1$.

Here $d_1,d_2\in\RR$ are the variables of the respective diagonal derivation $D$, the equations of $\Cone(\ngo_3)$ are $ 4d_1+d_2>0$ and $d_1+d_2>0,$ see Figure~\ref{cone3}. $\Cone(\ngo_5)$ is defined by $ d_1+2d_2>0$ and $2d_1+d_2>0$ as we can see in Figure~\ref{cone5}.

\subsection{Case $\ngo_4$} The Lie algebra $\ngo_4$ is the Heisenberg Lie algebra of dimension 5, and we will discuss it in Section~\ref{Heis}.

\newpage
    \begin{paracol}{2}
    \begin{figure}[h!]
\begin{tikzpicture}[scale=0.76]
\draw[very thin,color=gray,step=1cm] (-3.2,-3.2) grid (4.3,4.3);
\foreach \x in {-1,0,1,2,3}
   \draw (\x cm,1pt) -- (\x cm,-1pt) node[anchor=north] {$\x$};
\foreach \y in {0,1,2,3}
   \draw (1pt,\y cm) -- (-1pt,\y cm) node[anchor=east] {$\y$};

\draw[->] (-3,0) -- (4,0) node[right] {$d_1$};
\draw[->] (0,-2) -- (0,4) node[above] {$d_2$};

\draw[very thick, red]  (-2,4) -- (3/2,-3);

\draw[dashed, very thick, blue]  (0,0) -- (-1,4);
\draw[dashed, very thick, blue]  (0,0) -- (3,-3);

\draw[orange, fill, opacity=0.2]  (-1,4) .. controls (4,4) .. (3,-3) -- (0,0) -- (-1,4);

\draw (3.8,1.5) node[color=orange, thick] {$\Cone(\ngo_3)$};
\draw (-2,1.5) node[color=red,very thick] {$tr(D)>0$};
\draw (3.2,-3.2) node[color=blue,very thick] {$d_2=-d_1$};
\draw (0.1,3.7) node[color=blue,very thick] {$d_2=-4d_1$};
\end{tikzpicture}
\caption{$\Cone(\ngo_3)$}
\label{cone3}
\end{figure}
    \switchcolumn
\begin{figure}[h!]
\begin{tikzpicture}[scale=0.8]
\draw[very thin,color=gray,step=1cm] (-3.2,-3.2) grid (4.3,4.3);
\foreach \x in {-1,0,1,2,3}
   \draw (\x cm,1pt) -- (\x cm,-1pt) node[anchor=north] {$\x$};
\foreach \y in {0,1,2,3}
   \draw (1pt,\y cm) -- (-1pt,\y cm) node[anchor=east] {$\y$};

\draw[->] (-3,0) -- (4,0) node[right] {$d_1$};
\draw[->] (0,-2) -- (0,4) node[above] {$d_2$};

\draw[very thick, red]  (-3,3) -- (3,-3);

\draw[dashed, very thick, blue]  (0,0) -- (-2,4);
\draw[dashed, very thick, blue]  (0,0) -- (4,-2);

\draw[orange, fill, opacity=0.2]  (-2,4) .. controls (4,4) .. (4,-2) -- (0,0) -- (-2,4);

\draw (3.2,1.5) node[color=orange, thick] {$\Cone(\ngo_5)$};
\draw (-1.9,0.5) node[color=red,very thick] {$tr(D)>0$};
\draw (3.2,-0.7) node[color=blue,very thick] {$d_2=-\frac{1}{2}d_1$};
\draw (-0.5,3.5) node[color=blue,very thick] {$d_2=-2d_1$};
\end{tikzpicture}
\caption{$\Cone(\ngo_5)$}
\label{cone5}
\end{figure}
\end{paracol}

\subsection{Case $\ngo_6$}\label{n6}

A diagonal matrix $D$ is a derivation of $\ngo_6$ in $\tgo(\ngo_6)$ if and only if $D=\Diag(d_1,d_2,d_3,d_1+d_2,d_1+d_3), \text{ where } d_1,d_2,d_3\in \RR.$ In the generic case, i.e. when $d_1\neq d_2, \ d_1\neq d_3, \ d_2 \neq d_3, \ d_1+d_2\neq d_3, \ d_1+d_3\neq d_2 \text{ and } d_1,d_2,d_3 \neq 0,$ we obtain that $G_D=G_{\tgo(\ngo_6)}=\Diag(\ngo_6)_{>0}$, and $ \overline{ G_D \cdot \mu}$ is given by
$$\lambda(e_1,e_2)= \alpha e_4, \qquad \lambda(e_1,e_3)=\beta e_5.$$
where $\alpha, \beta \geq 0$.
As the basis is nice, by \cite[Lemma 3.14]{RN}, $\mm(\lambda)=\mm(\overline{G_D\cdot \mu})=\CH (F_{12}^4, F_{13}^5)$.

From \eqref{def cono} we say that $D \in \Cone(\ngo_6) \subset \tgo(\ngo_6)$ if and only if $D=aF_{12}^4+bF_{13}^5 + E$ and $\tr(D)>0$, where $a,b > 0$ and $E \in \Diag(\ngo_6)_{>0}$, this implies:
\begin{equation*}
\left\{
\begin{aligned}
  d_1+a+b & >0, \qquad  (6.1)\\
  d_2+a & >0, \qquad  (6.2)\\
  d_3+b & >0, \qquad  (6.3)\\
  \end{aligned}
  \right.
  \qquad \qquad
  \left\{
  \begin{aligned}
  d_1+d_2-a & >0,  \qquad (6.4)\\
  d_1+d_3-b & >0,  \qquad (6.5)
\end{aligned}
\right.
\end{equation*}

Note that, from conditions $(6.5)$, $(6.3)$, $(6.1)$, there exists $b>0$ such that
$$ d_1+d_3>b>-d_3 \ , \ -d_1-a,$$
and this holds if and only if $d_1+2d_3>0$ and by $(6.4)$ there exists $a>0$ such that
$$d_1+d_2>a>-2d_1-d_3 \ , \ -d_2, $$
this means $ 3d_1+d_2+d_3>0$ and $d_1+2d_2>0 $.
Then, the system above is equivalent to the next one.
$$ \left\{
\begin{aligned}
  d_1+d_2 & >0, \\
  d_1+d_3 & >0, \\
  d_1+2d_3 & >0, \\
  \end{aligned}
  \right.
  \qquad \qquad
  \left\{
  \begin{aligned}
  d_1+2d_2& >0, \\
  3d_1+d_2+d_3& >0.
\end{aligned}
\right. $$
As the cone $\Cone(\ngo_6)$ is invariant up to scaling, we can considerer $\tr(D)=1$ such that the equations of $\Cone_1(\ngo_6)$ are:
$$ \left\{
\begin{aligned}
 1-2d_1-d_2-2d_3 & >0, \\
 1-2d_1-2d_2-d_3 & >0,\\
 1-2d_1-2d_2 & >0, \\
   \end{aligned}
  \right.
  \qquad \qquad
  \left\{
  \begin{aligned}
 1-2d_1-2d_3& >0, \\
  1-d_2-d_3& >0.
\end{aligned}
\right. $$
We see that $\Cone_1(\ngo_6)$ is a five-sided polygon which its vertices are:
 $$\left(0,\tfrac{1}{2},0\right) \quad \left(0,0,\tfrac{1}{2}\right) \quad \left(1,-\tfrac{1}{2},-\tfrac{1}{2}\right) \quad \left(-\tfrac{1}{3},\tfrac{2}{3},\tfrac{1}{3}\right) \quad \left(-\tfrac{1}{3},\tfrac{1}{3},\tfrac{2}{3}\right) $$

We can see a figure of this if we consider the subspace $3d_1+2d_2+2d_3 =0$, an orthonormal basis of it and the coordinates of the vertices in this basis, see Figure~\ref{cone6}.

\begin{figure}[h]
\begin{tikzpicture}[scale=1.8]
\draw[very thin,color=gray,step=0.5cm] (-1.2,-1.2) grid (1.6,1.2);
\foreach \x in {-1,0,1}
   \draw (\x cm,1pt) -- (\x cm,-1pt) node[anchor=north] {$\x$};
\foreach \y in {-1,1}
   \draw (1pt,\y cm) -- (-1pt,\y cm) node[anchor=east] {$\y$};


\draw[->] (-1,0) -- (1.5,0) node[right] {$d_1$};
\draw[->] (0,-1) -- (0,1) node[above] {$d_2$};

\draw[very thick, blue, dashed]  (-0.208,-0.353) -- (1.249,0) -- (-0.208,0.353) -- (-0.694,0.236) -- (-0.694,-0.236) -- (-0.208,-0.353);

\draw[very thick, orange, fill, opacity=0.2]  (-0.208,-0.353) -- (1.249,0) -- (-0.208,0.353) -- (-0.694,0.236) -- (-0.694,-0.236) -- (-0.208,-0.353);
 \draw[color=blue, fill=blue]  (-0.208,-0.353) circle (0.8pt) node[below]{$(0,0,\tfrac{1}{2})$};
    \draw[color=blue, fill=blue]  (1.249,0) circle (.8pt) node[above]{$(1,-\tfrac{1}{2},-\tfrac{1}{2})$};
    \draw[color=blue, fill=blue]  (-0.208,0.353) circle (0.8pt) node[above]{$(0,\tfrac{1}{2},0)$};
    \draw[color=blue, fill=blue]  (-0.694,0.236) circle (0.8pt) node[above left]{$(-\tfrac{1}{3},\tfrac{2}{3},\tfrac{1}{3})$};
    \draw[color=blue, fill=blue]  (-0.694,-0.236) circle (0.8pt) node[below left]{$(-\tfrac{1}{3},\tfrac{1}{3},\tfrac{2}{3})$};

\end{tikzpicture}
\caption{$\Cone_1(\ngo_6)$}
\label{cone6}
\end{figure}

\subsection{Cases $\ngo_7$ and $\ngo_8$}\label{n78} For the cases $\ngo_7$ and $\ngo_8$, the equations of the cone $\Cone(\ngo)$ are $2d_1+d_2>0$, $d_1+2d_2>0$, $d_4>0$, $d_5>0$ and $2d_1+d_2>0$, $d_1+d_2>0$, $d_5>0$, respectively. Compare these results to \cite[Example 3.6 and Example 3.7]{RN}.

  \section{The validity of $\tgo(\ngo)_{srn}=\Cone(\ngo)$ in dimension 5}\label{theodim5}
  The goal of this section is to prove that the cone of all strongly Ricci negative derivations of nilpotent Lie algebras of dimension $5$ is equal to the cone defined in \cite{RNder} (see~\eqref{def cono}). First, we give three lemmas, which allow us to reduce the cases involved in the proof of the theorem.
\begin{lemma}\label{lem1}
  Let $\mathfrak{a},\mathfrak{b}$ be subspaces of $\tgo(\ngo)$. If there exists $f\in W_{ort}(\ngo)$ such that $f\mathfrak{a}f^{-1}=\mathfrak{b}$,
  then,
  \begin{center}
   $\mathfrak{a}\cap\tgo(\ngo)_{srn}\subset\Cone(\ngo)$ if and only if $\mathfrak{b}\cap\tgo(\ngo)_{srn}\subset\Cone(\ngo).$
  \end{center}
 \end{lemma}
 \begin{proof}
 Suppose that $D \in \mathfrak{b}\cap\tgo(\ngo)_{srn}$, then $f^{-1}Df \in \mathfrak{a}\cap\tgo(\ngo) \subset \Cone(\ngo)$, because of the hypothesis.
 As $\Cone(\ngo)$ is $W_{ort}$- invariant, we have that $D \in f\Cone(\ngo)f^{-1}=\Cone(\ngo)$.
  \end{proof}

  \begin{lemma}\label{lem2}
    Let $\ngo$ be a nilpotent Lie algebra with Lie bracket $\mu$. If $\Cone(\ngo)$ is the cone defined in \eqref{def cono}, then
    $$\Cone(\ngo\oplus\RR)=\Cone(\ngo)\times\RR_{>0}.$$
  \end{lemma}
 \begin{proof}
 Let $D\in\tgo(\ngo\oplus\RR)_{gen}$; that is $D|_{\ngo}=D_1$ and $D|_{\RR}=d$, where $D_1 \in \tgo(\ngo)_{gen}$ and $d\in\RR$ such that $d\notin\Spec(D_1)$, note that $G_{\tgo(\ngo)}=G_{D_1}$.
    In this case $G_D=G_{\tgo(\ngo\oplus\RR)}$ is given by  $h\in\Gl(\ngo\oplus\RR)$ such that $h|_{\ngo}\in G_{D_1}$ and $h|_{\RR}=a$ with $a>0.$

 As $G_D\cdot\mu$ is given by the brackets defined by $G_{D_1}\cdot\mu$, we have that $\mm(\overline{G_D\cdot\mu})|_{\ngo}=\mm(\overline{G_{D_1}\cdot\mu})$ and $\mm(\overline{G_D\cdot\mu})|_{\RR}=0$.
By \eqref{def cono}, $D\in\Cone(\ngo\oplus\RR)$ if and only if $ D\in\RR_{>0}\mm(\overline{G_D\cdot\mu})\cap \Diag(\ngo\oplus\RR)+\Diag(\ngo\oplus\RR)_{>0}$, i.e. if and only if there exist $r,t>0$, $h\in G_D$ and $E_1\in \Diag(\ngo)_{>0}$ such that
            $$D_1= r\mm(h\cdot\mu)+E_1, \qquad d = t.$$

 Therefore $\Cone(\ngo\oplus\RR)=\Cone(\ngo)\times\RR_{>0}$.
  \end{proof}

\begin{lemma}\label{lem3}
Given $\ngo$ a nilpotent Lie algebra, suppose that $\tgo(\ngo)_{srn}=\Cone(\ngo)$. If $D|_{\ngo}=D_1$ and $D|_{\RR}=d$ such that
$d\notin\Spec(D_1)$, then $D\in\Cone(\ngo\oplus\RR)$.
 \end{lemma}
 \begin{proof}
 By Lemma~\ref{lem2} we already know that
 $$\Cone(\ngo\oplus\RR)=\Cone(\ngo)\times\RR_{>0}=\tgo(\ngo)_{srn}\times\RR_{>0}$$
 Given  $D\in \tgo(\ngo\oplus\RR)_{srn}$ under the hypothesis, analysis similar to that in the proof of Lemma~\ref{lem2} shows that $D\in\tgo(\ngo)_{srn}\times\RR_{>0}.$
  \end{proof}

  \begin{theorem}\label{dim5}
 For any nilpotent Lie algebra $\ngo$ of dimension $5$, we have that
  \begin{center}
    $\tgo(\ngo)_{srn}=\Cone(\ngo)$.
  \end{center}
  \end{theorem}

 \begin{proof}
    The proof is completed by showing that all the strongly Ricci negative non generic derivations of  $\ngo$ belong to the cone $\Cone(\ngo)$. This means that if a non generic derivation of $\ngo$ does not belong to $\Cone(\ngo)$, then it can not be strongly Ricci negative.
    As $\Cone(\ngo)$ is convex, we already know that all the derivations inside of it are also in $\tgo(\ngo)_{srn}$.

    First, consider $\ngo_1$ the nilpotent Lie algebra defined in Table~\ref{clasifi}, according to Section~\ref{n1}, we have that the only non generic derivation with positive trace of $\ngo_1$, which does not belong to the cone $\Cone{(\ngo_1)}$ is (up to scaling)
    $ D_0= \Diag(-1,2,1,0,-1)$. We know from \cite[Theorem 2]{NklNkn} that it can not be strongly Ricci negative because $D_0|_{\zg(\ngo_1)}=-1<0.$

    In the case $\ngo_2$, the only non generic derivation vanishes and it has not positive trace, so it can not be in $\tgo(\ngo_2)_{srn}$.

    Taking the case $\ngo_5$, we can see that all the non generic derivations of positive trace are in $\Cone(\ngo_5)$ which means that they are strongly Ricci negative.

    So, until now we see that $\Cone(\ngo_j)=\tgo(\ngo_j)_{srn} \text{ for } j=1,2,5.$ The remaining cases are more difficult and we give the proof only for the case $\ngo_6$; the other cases follow by the same method. For the cases $\ngo_7$ and $\ngo_8$ Lemmas \ref{lem2} and \ref{lem3} are very useful.

    Applying the results of Section~\ref{n6}, we see that $D$ is a derivation of $\ngo_6$ if and only if $ D= \Diag(d_1,d_2,d_3,d_1+d_2,d_1+d_3)$, the non generic cases happen when some of the following equalities holds:
          $$d_1=0; \ d_2=0; \ d_3=0; \ d_1= d_2; \ d_1=d_3; \ d_2=d_3; \ d_1+d_2=d_3; \ d_1+d_3=d_2.$$
    Define $\mathfrak{a}=\{D\in\tgo(\ngo_6):d_1=0\}$. Any derivation of this subspace is given by $D_{\mathfrak{a}}=\Diag(0,d_2,d_3,d_2,d_3)$, the connected component of the identity of the centralizer subgroup of $D_{\mathfrak{a}}$ in $GL(\ngo_6)$, denoted $G_{D_{\mathfrak{a}}}$, is defined by
          $$ H= \left(\begin{smallmatrix}
            t &  &  &  &  \\
             & a & & b &   \\
             &  & p &  &q  \\
             & c &  &d &  \\
             &   &r &  & s
          \end{smallmatrix}\right), \text{ where } \ det_1:=ad-bc\neq0 \text{ and } det_2:=ps-qr\neq 0.$$
    Since
    \begin{align*}
          &H\cdot\mu(e_1,e_2)=H\mu(\tfrac{1}{t}e_1,\tfrac{d}{det_1}e_2-\tfrac{c}{det_1}e_4)=\tfrac{bd}{tdet_1}e_2+\tfrac{d^2}{tdet_1}e_4,\\
          &H\cdot\mu(e_1,e_3)=H\mu(\tfrac{1}{t}e_1,\tfrac{s}{det_2}e_3-\tfrac{r}{det_2}e_5)=\tfrac{qs}{tdet_2}e_3+\tfrac{s^2}{tdet_2}e_5,\\
          &H\cdot\mu(e_1,e_4)=H\mu(\tfrac{1}{t}e_1,-\tfrac{b}{det_1}e_2+\tfrac{a}{det_1}e_4)=-\tfrac{b^2}{tdet_1}e_2 - \tfrac{bd}{tdet_1}e_4,\\
          &H\cdot\mu(e_1,e_5)=H\mu(\tfrac{1}{t}e_1,-\tfrac{q}{det_2}e_3+\tfrac{p}{det_2}e_5)=-\tfrac{q^2}{tdet_2}e_3-\tfrac{qs}{tdet_2}e_3,
          \end{align*}
     the bracket $\overline{G_{D_{\mathfrak{a}}}\cdot\mu}$ is defined by:
          $$ \lambda(e_1,e_2)=\alpha e_2+\beta e_4, \quad \lambda(e_1,e_3)=\gamma e_3+\xi e_5,$$
           $$ \lambda(e_1,e_4)=\tau e_2-\alpha e_4, \quad \lambda(e_1,e_5)=\iota e_3 - \gamma e_5,$$
     where all of the constants are non negative.
     From the definition of moment map (see \eqref{defmm2}) we have that:
          $$\mm(\lambda)=\frac{2}{|\lambda|^2}\left( \begin{smallmatrix}
            -2\alp -\bet -2\gam -\xi^2-\tau^2 - \iota^2 &  &  &  &  \\
             & - \bet+\tau^2   &  &M_{24} &    \\
             &  & -\xi^2+\iota^2  &  & M_{35} \\
             &  M_{24} &  & +\bet-\tau^2 &  \\
             &    &   M_{35} & & \xi^2-\iota^2
          \end{smallmatrix}
          \right)$$
          where $M_{24}=2\alpha(\beta-\tau)$ and $ M_{35}=2\gamma(\xi-\iota).$
          So $\mm(\overline{G_{D_{\mathfrak{a}}}\cdot\mu})\cap\Diag(\ngo_6)\neq{\emptyset}$ if and only if $bd=0$ and $qs=0$, which means
            $[\alpha=0 \ \wedge \ (\tau=0$ or  $\beta=0)]$ and $[\gamma=0 \ \wedge \ (\iota=0$ or  $\xi=0)]$.
          From all of this, we have that $\mm(\overline{G_{D_{\mathfrak{a}}}\cdot\mu})\cap\Diag(\ngo_6)=\CH(F_{12}^4,F_{13}^5)\cup\CH(F_{12}^4,F_{15}^3)\cup\CH(F_{14}^2,F_{13}^5)\cup\CH(F_{14}^2,F_{15}^3)$,
and from Theorem~\ref{main} (iv), ${D_{\mathfrak{a}}}\in\mathfrak{a}\cap\tgo(\ngo_6)_{srn}$ if and only if ${D_{\mathfrak{a}}}=aF_{14}^2+bF_{15}^3+ E$ where $a,b >0$ and $E\in\Diag(\ngo_6)_{>0}$. This holds if and only if $d_2>0$ and $d_3>0$, as these are equations of $\Cone(\ngo_6)\cap\mathfrak{a}$ we conclude that $\mathfrak{a}\cap\tgo(\ngo_6)_{srn}\subset\Cone(\ngo_6)$.

To analyze the non generic cases $d_2=0$ and $d_3=0$, we define $\mathfrak{b}=\{D\in\tgo(\ngo_6):d_2=0\} \text{ and } \mathfrak{c}=\{D\in\tgo(\ngo_6):d_3=0\}$.
It is evident that there exists $f \in W_{ort}(\ngo_6)$ such that $\mathfrak{b}=f^{-1}\mathfrak{c}f$. By Lemma~\ref{lem1}, if we prove the case $d_2=0$, the case $d_3=0$ follows. Let $ D_{\mathfrak{b}}= \Diag(d_1,0,d_3,d_1,d_1+d_3) \in \mathfrak{b}\cap\tgo(\ngo_6)$, $G_{D_{\mathfrak{b}}}$ is given by:
          $$ H= \left(\begin{smallmatrix}
            a & &   & b  &  \\
             & p & &  &  \\
             &  & q &  &  \\
            c & & & d &  \\
             &  & & &  & r
          \end{smallmatrix}\right), \text{ where }  det:=ad-bc\neq0. $$

          $\overline{G_{D_{\mathfrak{b}}}\cdot\mu}$ is defined by:
          $$ \lambda(e_1,e_2)=\alpha e_1+\beta e_4, \quad \lambda(e_1,e_3)=\gamma e_5,$$ $$\lambda(e_2,e_4)=\xi e_1+ \alpha e_4, \quad \lambda(e_3,e_4)=\tau e_5,$$
          because we can compute
            \begin{align*}
          &H\cdot\mu(e_1,e_2)=H\mu(\tfrac{d}{det}e_1-\tfrac{c}{det}e_4,\tfrac{1}{p}e_2)=\tfrac{db}{pdet}e_1+\tfrac{d^2}{pdet}e_4,\\
          &H\cdot\mu(e_1,e_3)=H\mu(\tfrac{d}{det}e_1-\tfrac{c}{det}e_4,\tfrac{1}{q}e_3)=\tfrac{dr}{qdet} e_5,\\
          &H\cdot\mu(e_2,e_4)=H\mu(\tfrac{1}{p}e_2,-\tfrac{b}{det}e_1+\tfrac{a}{det}e_4)=\tfrac{b^2}{pdet}e_1+\tfrac{bd}{pdet}e_4,\\
          &H\cdot\mu(e_3,e_4)=H\mu(\tfrac{1}{q}e_3,-\tfrac{b}{det}e_1+\tfrac{a}{det}e_4)=\tfrac{br}{qdet} e_5.
          \end{align*}
          From \eqref{defmm2} we have:
          $$\mm(\lambda)=\frac{2}{|\lambda|^2} \left(\begin{smallmatrix}
          -\bet+\xi^2-\gam  & &   & M_{14} &    \\
             & -2\alp-\bet-\xi^2   & &   &    \\
                & & -\gam-\tau^2  &    \\
             M_{14}&   & & \bet-\tau^2-\xi^2 &  \\
             & &   &  & & \gam+\tau^2
          \end{smallmatrix}
          \right)$$
          where $M_{14}=2\alpha\xi+2\beta\alpha+\gamma\tau$. We can see that $\mm(\overline{G_{D_{\mathfrak{b}}}\cdot\mu})\cap\Diag(\ngo_6)\neq{\emptyset}$ if and only if $M_{14}=bd(\tfrac{2b^2}{p^2det^2}+\tfrac{2d^2}{p^2det^2}+\tfrac{r^2}{q^2})=0$. So, $\mm(\overline{G_{D_{\mathfrak{b}}}\cdot\mu})\cap\Diag(\ngo_6)\neq{\emptyset}$ if and only if $\alpha=0 \wedge [(\beta=0 \text{ and } \gamma=0)$ or $(\xi=0  \text{ and } \tau=0)].$
          As we have $\mm(\overline{G_{D_{\mathfrak{b}}}\cdot\mu})\cap\Diag(\ngo_6)=\CH(F_{24}^1,F_{34}^5)\cup\CH(F_{12}^4,F_{13}^5)$, by Theorem~\ref{main} (iv), $D_{\mathfrak{b}}\in\tgo(\ngo_6)_{srn}$ if and only if $D_{\mathfrak{b}}=aF_{24}^1+bF_{34}^5+ E$ with $a,b >0$ and $E\in\Diag(\ngo_6)_{>0}$. We can see that this is equivalent to have $d_1+2d_3>0$, $3d_1+d_3>0$ and $d_1>0$, which are equations of $\Cone(\ngo_6)\cap\mathfrak{b}$, so $\mathfrak{b}\cap\tgo(\ngo_6)_{srn}\subset\Cone(\ngo_6)$.

    We can now procced analogously to the proof of the non generic cases: $d_1=d_2$, $d_1=d_3$, $d_1+d_3=d_2$, $d_1+d_2=d_3$, $d_2=d_3$ and the intersections between them.
    Using the same $f\in W_{ort}(\ngo_6)$, we define $\mathfrak{d}=\{D\in\tgo(\ngo_6):d_1=d_2\},$ $\mathfrak{j}=\{D\in\tgo(\ngo_6):d_1=d_3\}, \ \mathfrak{k}=\{D\in\tgo(\ngo_6):d_1+d_3=d_2\}, \ \mathfrak{l}=\{D\in\tgo(\ngo_6):d_1+d_2=d_3\}$ and we say that $\mathfrak{d}=f^{-1}\mathfrak{j}f$ and  $\mathfrak{k}=f^{-1}\mathfrak{l}f$.
    After these cases we deal with the case $d_2=d_3$.

     Let $ D_{\mathfrak{d}}= \Diag(d_1,d_1,d_3,2d_1,d_1+d_3)\in \mathfrak{d}\cap\tgo(\ngo_6)$. $G_{D_{\mathfrak{d}}}$ is given by
          $$ H= \left(\begin{smallmatrix}
            a & b &  &   \\
            c & d & & &  &  \\
             &  & p &  &  \\
             &  &   & q &  \\
             &  & & &   r
          \end{smallmatrix}\right), \text{ with } det:=ad-bc\neq0.$$
    It is easy to check that $\overline{G_{D_{\mathfrak{d}}}\cdot\mu}$ is given by:
          $$ \lambda(e_1,e_2)=\alpha e_4, \quad \lambda(e_2,e_3)=\beta e_5,\quad \lambda(e_1,e_3)=\gamma e_5.$$
    We can compute the moment map of $\lambda$ following \eqref{defmm2},
           $$\mm(\lambda)=\frac{2}{|\lambda|^2} \left( \begin{smallmatrix}
            -\alp-\gam  & M_{12} & & &    \\
             M_{12}& -\alp-\bet  & & &    \\
             &   & -\bet-\gam  &  &  \\
             & & & \alp &  \\
             & &   &  & +\bet+\gam
          \end{smallmatrix} \right) \text{ where }M_{12}=-\gamma\beta.$$

    So, $\mm(\overline{G_{D_{\mathfrak{d}}}\cdot\mu})\cap\Diag(\ngo_6)=\CH(F_{12}^4,F_{23}^5)\cup\CH(F_{12}^4,F_{13}^5).$
   By Theorem~\ref{main} (iv), $D_{\mathfrak{d}}\in\tgo(\ngo_6)_{srn}$ if and only if there exist $a,b>0$, and $E \in \Diag(\ngo_6)_{>0}$ such that $D_{\mathfrak{d}}=aF_{12}^4+bF_{23}^5+E$, this equality holds if and only if $3d_1>0$, $d_1+d_3>0$, $d_1+2d_3>0$ and $4d_1+d_3>0$. Because of this, we claim that $\mathfrak{d}\cap\tgo(\ngo_6)_{srn}\subset\Cone(\ngo_6)$, and by Lemma~\ref{lem1} we can say $\mathfrak{j}\cap\tgo(\ngo_6)_{srn}\subset\Cone(\ngo_6)$.

   Taking $ D_{\mathfrak{k}}= \Diag(d_1,d_2,d_3,d_1+d_2,d_2) \in \mathfrak{k}\cap\tgo(\ngo_6)$, we can assert that $G_{D_\mathfrak{k}}$ is defined by:
          $$ H=\left(\begin{smallmatrix}
            p &  &  &  &  \\
             & a & & &  b  \\
             &  & q &  &  \\
             &  &  & r &  \\
             &  c & &  & d
          \end{smallmatrix}\right), \text{ where } det:=ad-bc\neq0.$$
    If we compute
          \begin{align*}
          &H\cdot\mu(e_1,e_2)=H\mu(\tfrac{1}{p}e_1,\tfrac{d}{det}e_2-\tfrac{c}{det}e_5)=\tfrac{d}{pdet}re_4,\\
          &H\cdot\mu(e_1,e_3)=H\mu(\tfrac{1}{p}e_1,\tfrac{1}{q}e_3)=\tfrac{1}{pq}(be_2+de_5),\\
          &H\cdot\mu(e_1,e_5)=H\mu(\tfrac{1}{p}e_1,-\tfrac{b}{det}e_2+\tfrac{a}{det}e_5)=-\tfrac{b}{pdet}re_4.
          \end{align*}
    We obtain that
    $$ \lambda(e_1,e_2)=\alpha e_4, \quad \lambda(e_1,e_3)=\beta e_2+\gamma e_5,\quad \lambda(e_1,e_5)=\xi e_4,$$
    define $\overline{G_{D_\mathfrak{k}}\cdot\mu}$.
    From the definition of the moment map (see \eqref{defmm2}) we have:
          $$\mm(\lambda)=\frac{2}{|\lambda|^2} \left( \begin{smallmatrix}
            -\alp -\bet -\gam-\xi^2 &  &  &  &  \\
             & -\alp+\bet   &  & &  M_{25}  \\
             &  & -\bet-\gam  &  &  \\
             &  &  & +\alp+\xi^2 &  \\
             &  M_{25}  &  & & +\gam-\xi^2
          \end{smallmatrix}
          \right)$$
          where $ M_{25}=-\alpha\xi+\beta\gamma =\tfrac{r^2}{p^2det^2}bd+\tfrac{1}{p^2q^2}bd=bd(\tfrac{r^2}{p^2det^2}+\tfrac{1}{p^2q^2}).$

          Then $\mm(\overline{G_{D_\mathfrak{k}}\cdot\mu})\cap\Diag(\ngo_6)\neq{\emptyset}$ if and only if $bd=0$, which implies
            $\beta=0  \wedge \xi=0$ or $\alpha=0 \wedge \gamma=0$. Therefore, $\mm(\overline{G_{D_\mathfrak{k}}\cdot\mu})\cap\Diag(\ngo_6)=\CH(F_{12}^4,F_{13}^5)\cup\CH(F_{13}^2,F_{15}^4)$.

          By Theorem~\ref{main} (iv), $D_\mathfrak{k}\in\tgo(\ngo_6)_{srn}$ if and only if $D_\mathfrak{k}=aF_{13}^2+bF_{15}^4+ E$ where $a,b >0$ and $E\in\Diag(\ngo_6)_{>0}$.
          This holds if and only if $d_1+d_2>0$, $d_1+2d_2>0$, $d_2+d_3>0$ and $2d_1+2d_2>0$.
        These are equations of $\Cone(\ngo_6)\cap\mathfrak{k}$, so $\mathfrak{k}\cap\tgo(\ngo_6)_{srn}\subset\Cone(\ngo_6)$ and $\mathfrak{l}\cap\tgo(\ngo_6)_{srn}\subset\Cone(\ngo_6)$.

        We now turn to the case $d_2=d_3$. Define $\mathfrak{p}=\{D\in\tgo(\ngo_6): d_2=d_3\}$ and $ D_{\mathfrak{p}}= \Diag(d_1,d_2,d_2,d_1+d_2,d_1+d_2) \in \mathfrak{p} \cap\tgo(\ngo_6)$.
        As $G_{D_{\mathfrak{p}}}$ is given by:
          $$ H= \left(\begin{smallmatrix}
            t &  &  &    \\
             & a & b & &  &  \\
             &  c& d &  &  \\
             &  &  &p&q  \\
             &  & & r&  s
          \end{smallmatrix}\right), \text{ with } det_1:=ad-bc\neq0 \text{ and } det_2:=ps-qr\neq0,$$

       it is immediate that $\overline{G_{D_{\mathfrak{p}}}\cdot\mu}$ is defined by:
          $$ \lambda(e_1,e_2)=\alpha e_4+\beta e_5, \quad \lambda(e_1,e_3)=\gamma e_4+ \xi e_5.$$
       and
          $$\mm(\lambda)=\frac{2}{|\lambda|^2} \left(\begin{smallmatrix}
            -\alp-\bet-\gam-\xi^2  & & & &      \\
             & -\alp-\bet  &M_{23} & &     \\
             &   M_{23} & -\gam-\xi^2  &    \\
             &  & & +\alp+\gam &M_{45}  \\
             &    &  & M_{45}& \bet+\xi^2
          \end{smallmatrix}\right)$$
           where $M_{23}=-\alpha\gamma-\beta\xi \text{ and } M_{45}=\alpha\beta+\gamma\xi.$

       So, $\mm(\overline{G_{D_{\mathfrak{p}}}\cdot\mu})\cap\Diag(\ngo_6)\neq{\emptyset}$ if and only if $(\alpha,\xi)\bot(\gamma,\beta)$ and $(\alpha,\xi)\bot(\beta,\gamma)$ which implies $\alpha=0 \wedge \xi=0$ or $\gamma=\pm\beta$.
       Applying this, $\mm(\overline{G_{D_{\mathfrak{p}}}\cdot\mu})\cap\Diag(\ngo_6)=\CH(F_{12}^5,F_{13}^4)\cup\CH(F_{12}^4,F_{12}^5,F_{13}^4,F_{13}^5)$.

       By Theorem~\ref{main} (iv), $D_{\mathfrak{p}}\in\tgo(\ngo_6)_{srn}$ if and only if there are $a,b>0$, and $E \in \Diag(\ngo_6)_{>0}$ such that $D_{\mathfrak{p}}=aF_{12}^5+bF_{13}^4+ E$, which holds if and only if $d_1+d_2>0$, $d_1+2d_2>0$ and $3d_1+2d_2>0$. Because of what we have seen in Section \ref{conosdim5}, these are equations of $\Cone(\ngo_6)\cap\mathfrak{p}$, therefore $\mathfrak{p}\cap\tgo(\ngo)_{srn}\subset\Cone(\ngo_6)$.

       We summarize in Table~\ref{tablen6} the results considering the intersection of the subspaces defined.

           \begin{table}[h]
          \renewcommand{\arraystretch}{1.5}
        \begin{tabular}{|p{3cm}|p{4.8cm}|p{7.2cm}|}
          \hline
          Subspaces & Derivation & Conclusion \\
          \hline \hline
          $\mathfrak{a}\cap\mathfrak{b}$, \ $\mathfrak{a}\cap\mathfrak{d}$, \ $\mathfrak{b}\cap\mathfrak{d}$ & $ D= \Diag(0,0,d_3,0,d_3)$ & There is a $0$ in the center so, $D\notin\tgo(\ngo_6)_{srn}$.  \\
          \hline
          $\mathfrak{a}\cap\mathfrak{c}$, \ $\mathfrak{a}\cap\mathfrak{j}$, \ $\mathfrak{c}\cap\mathfrak{j}$ & $ D= \Diag(0,d_2,0,d_2,0)$ & There is a $0$ in the center so, $D\notin\tgo(\ngo_6)_{srn}$. \\
          \hline
          $\mathfrak{a}\cap\mathfrak{k}$,  \ $\mathfrak{a}\cap\mathfrak{l}$, \ $\mathfrak{a}\cap\mathfrak{p}$, \ $\mathfrak{k}\cap\mathfrak{l}$ \ $\mathfrak{k}\cap\mathfrak{p}$, \ $\mathfrak{l}\cap\mathfrak{p}$,  & $ D= \Diag(0,d_2,d_2,d_2,d_2)$ & By nec condition $d_2>0$ and so, $D\in\Cone(\ngo_6)$ \\
          \hline
          $\mathfrak{b}\cap\mathfrak{c}$, \ $\mathfrak{b}\cap\mathfrak{p}$, \ $\mathfrak{c}\cap\mathfrak{p}$ & $ D= \Diag(d_1,0,0,d_1,d_1)$ & By nec condition  $d_1>0$ and so, $D\in\Cone(\ngo_6)$ \\
                  \hline
          $\mathfrak{b}\cap\mathfrak{j}$, \ $\mathfrak{b}\cap\mathfrak{l}$, \ $\mathfrak{j}\cap\mathfrak{l}$ & $ D= \Diag(d_1,0,d_1,d_1,2d_1)$ & By nec condition $d_1>0$ and so, $D\in\Cone(\ngo_6)$ \\
                  \hline
            $\mathfrak{b}\cap\mathfrak{k}$ & $ D= \Diag(d_1,0,-d_1,d_1,0)$ & There is a $0$ in the center so, $D\notin\tgo(\ngo_6)_{srn}$. \\
                  \hline
          $\mathfrak{c}\cap\mathfrak{d}$ \ $\mathfrak{c}\cap\mathfrak{k}$ \ $\mathfrak{d}\cap\mathfrak{k}$ & $ D= \Diag(d_1,d_1,0,2d_1,d_1)$ & By nec condition $d_1>0$ and so, $D\in\Cone(\ngo_6)$ \\
          \hline
          $\mathfrak{c}\cap\mathfrak{l}$ & $ D= \Diag(d_1,-d_1,0,0,d_1)$ &  There is a $0$ in the center so, $D\notin\tgo(\ngo_6)_{srn}$. \\
                  \hline
          $\mathfrak{d}\cap\mathfrak{j}$, \ $\mathfrak{d}\cap\mathfrak{p}$, \ $\mathfrak{j}\cap\mathfrak{p}$ & $ D= \Diag(d_1,d_1,d_1,2d_1,2d_1)$ & By nec condition $d_1>0$ and so, $D\in\Cone(\ngo_6)$ \\
                  \hline
          $\mathfrak{d}\cap\mathfrak{l}$ & $ D=\Diag(d_1,d_1,2d_1,2d_1,3d_1)$ & By nec condition $d_1>0$ and so, $D\in\Cone(\ngo_6)$  \\
                  \hline
          $\mathfrak{j}\cap\mathfrak{k}$ & $ D=\Diag(d_1,2d_1,d_1,3d_1,2d_1)$ & By nec condition  $d_1>0$ and so, $D\in\Cone(\ngo_6)$  \\
                  \hline
          \hline
        \end{tabular}
        \caption{Case $\ngo_6$}
        \label{tablen6}
        \end{table}

      Hence, we obtain that $\Cone(\ngo_6)=\tgo(\ngo_6)_{srn}$. The remaining cases can be proved in much the same way.
   \end{proof}

  \section{Heisenberg Lie algebra}\label{Heis}
  In this section we proceed with the study of $\Cone(\ngo)$ defined in \eqref{def cono} for the Heisenberg Lie algebra endowed with the standard basis. Results of \cite{NklNkn} are useful to prove the Conjecture 1 for this Lie algebra. We begin by introducing the following proposition.

    \begin{proposition}\label{prop heis}
  Let $\mathfrak{h}_{2n+1}$ be the Heisenberg Lie algebra with basis $\{e_1,e_2,\ldots, e_{2n+1}\}$ and Lie brackets:
  $$ \mu(e_1,e_2)=e_{2n+1}, \quad \mu(e_3,e_4)=e_{2n+1}, \quad \ldots, \quad \mu(e_{2n-1},e_{2n})=e_{2n+1}.$$
  For $ D:= \Diag(d_1,d_{n+1}-d_1,d_2,d_{n+1}-d_2,\ldots,d_n,d_{n+1}-d_n,d_{n+1})\in \tgo(\mathfrak{h}_{2n+1}),$ the $3^n$ equations which define the cone $\Cone(\mathfrak{h}_{2n+1})$ are given by:
  \begin{equation}\label{conoheis}
  (l+1)d_{n+1}\pm d_{i_1} \pm d_{i_2} \pm \ldots \pm d_{i_k} >0, \quad 0\leq k \leq n, \quad 0\leq l \leq k.
  \end{equation}
  \end{proposition}
  \begin{remark}
  For each $k$ there are $ 2^k{n \choose k} $ equations which are all the ways to choose $l$ minus signs in the formula with $l=0,\dots k.$
  \end{remark}
\begin{proof}
  Let $\mathfrak{h}_{2n+1}$ be the Heisenberg Lie algebra. Let $D$ be a diagonal derivation of $\mathfrak{h}_{2n+1}$, that is:
           $$ D= \Diag(d_1,d_{n+1}-d_1,d_2,d_{n+1}-d_2,\ldots,d_n,d_{n+1}-d_n,d_{n+1}).$$
Note that $ \tr(D)=(n+1)d_{n+1}$ and  $D|_{\zg(\mathfrak{h}_{2n+1})}=d_{n+1}$.
$D\in\tgo(\mathfrak{h}_{2n+1})_{gen}$ if and only if
  $$d_i \neq 0 \quad \ i=1 \ldots n+1, \qquad d_i \neq d_{n+1} \quad  \ i=1\ldots n, \qquad d_i \neq d_j \quad \text{if} \ i\neq  j, $$
  $$ d_i+d_j \neq d_{n+1} \quad  \ i,j=1 \ldots n, \qquad i\neq j.$$
In this case $D$ has pairwise diﬀerent eigenvalues, and $\tgo(\mathfrak{h}_{2n+1})$ is multiplicity-free, because of that $G_{\tgo(\mathfrak{h}_{2n+1})}=\Diag(\mathfrak{h}_{2n+1})_{>0}$ and $ \overline{\Diag(\mathfrak{h}_{2n+1})_{>0}\cdot \mu}$ is given by:
  $$ \lambda(e_1,e_2)=\alpha_1 e_{2n+1}, \quad \lambda(e_3,e_4)=\alpha_2 e_{2n+1}, \quad \ldots, \quad \lambda(e_{2n-1},e_{2n})=\alpha_n e_{2n+1},$$
  where $\alpha_i \geq 0, \ i=1\ldots n $.
  We have that
   {\small
          \begin{align*}
          \mm(\lambda)&=\frac{2}{|\lambda|^2}\Diag(-\alp_1,-\alp_1,-\alp_2,-\alp_2,\ldots,-\alp_n,-\alp_n,\alp_1+\alp_2+\ldots+\alp_n)\\
          &= \frac{2}{|\lambda|^2}(\alp_1 F_{12}^{2n+1}+\alp_2 F_{34}^{2n+1}+\ldots +\alp_n F_{2n-1 2n}^{2n+1}).
          \end{align*}
          }
   And from \eqref{def cono}, $D$ belongs to $\Cone(\mathfrak{h}_{2n+1})$ if and only if there exist $a_1,a_2,\ldots, a_n >0$ and $ E \in \Diag(\mathfrak{h}_{2n+1})_{>0}$ such that
   $$ D=a_1F_{12}^{2n+1}+a_2F_{34}^{2n+1}+\ldots +a_nF_{2n-12n}^{2n+1} + E.$$
   We therefore obtain
   \begin{equation*}
               \left\{
             \begin{aligned}
              d_1+a_1& >0,  \\
              d_{n+1}-d_1+a_1& >0,\\
              d_2+a_2& >0,  \\
              d_{n+1}-d_2+a_2& >0,\\
              \end{aligned}
              \right.
              \qquad \qquad
              \left\{
              \begin{aligned}
              \vdots \\
              d_n+a_n& >0,  \\
              d_{n+1}-d_n+a_n& >0,\\
              d_{n+1}-a_1-a_2-\ldots - a_n& >0.\\
            \end{aligned}
            \right.
   \end{equation*}
  In order to obtain the equations defining $\Cone(\mathfrak{h}_{2n+1})$, we note that $d_{n+1}>0$ and there exists $a_1>0$ such that
  $$d_{n+1}-a_2-\ldots -a_n > a_1> -d_1, \quad d_1-d_{n+1}$$
  This means that the last system of inequalities implies $d_{n+1}+d_1>0$ and $2d_{n+1}-d_1>0$, the same is valid for all $d_i$ with $i=1,\ldots n$ if we replace $a_1>0$ by $a_i>0$.
  After this is done, we can see that there is $a_2>0$ such that
  $$ d_1+d_{n+1}-a_3 -\ldots -a_n > a_2 > -d_2, \quad d_2-d_{n+1},$$
  $$ 2d_{n+1}-d_1-a_3- \ldots - a_n > a_2 > -d_2,\quad d_2-d_{n+1},$$
  these inequalities implies $d_{n+1}+d_1+d_2>0$, $2d_{n+1}+d_1-d_2>0$, $2d_{n+1}-d_1+d_2>0$ and $3d_{n+1}-d_1-d_2>0$. As we do before, we can replace $d_1,d_2$ by $d_i,d_j$ $i\neq j$ and those equations are implied from the system too. We also have that,
  $$ d_{n+1}+d_1+d_2-a_3-\ldots -a_n > 0, \qquad \qquad 2d_{n+1}+d_1-d_2-a_3-\ldots -a_n >0,$$
  $$ 2d_{n+1}-d_1+d_2-a_3- \ldots- a_n >0, \qquad \qquad 3d_{n+1}-d_1-d_2-a_3- \ldots -a_n>0,$$
  and if we continue in the same manner with $a_3>0$ or other one, until we get inequalities with no positive constants involved we obtain $3^n-1$ inequalities, these and $d_{n+1}>0$ define the cone $\Cone(\mathfrak{h}_{2n+1})$, and this is precisely the assertion of the proposition.
\end{proof}

Now, we are going to study the polytope $\Cone_p(\mathfrak{h}_{2n+1}):=\Cone(\mathfrak{h}_{2n+1})\cap \{\tr(D)=p\}$ in low dimensions.

\subsection {$\Cone_3(\mathfrak{h}_{5})$}
  In $\mathfrak{h}_5$ we have $\tgo(\mathfrak{h}_5)$ defined by $D=\Diag(d_1,d_3-d_1,d_2,d_3-d_2,d_3),$ where $d_1,d_2,d_3 \in \RR$. As $\tr(D)=3d_3$, if we consider $\Cone_3(\mathfrak{h}_5)$, we have $d_3=1$ and it follows from Proposition \ref{prop heis} that the equations of $\Cone_3(\mathfrak{h}_5)$ are:
   \begin{equation*}
             \begin{aligned}
              1+d_1 & >0,\\
              1+d_2 & >0,\\
              2-d_1 & >0,\\
              2-d_2 & >0,
            \end{aligned}
              \qquad \qquad
              \begin{aligned}
              1+d_1+d_2& >0,  \\
              2+d_1-d_2& >0,\\
              2-d_1+d_2& >0,  \\
              3-d_1-d_2& >0.\\
            \end{aligned}
    \end{equation*}

    In Figure~\ref{figh5} we can see a representation of $\Cone_3(\mathfrak{h}_5).$

  \subsection { $\Cone_4(\mathfrak{h}_{7})$}
   In this case we have that $\tgo(\mathfrak{h}_7)$ is defined by
  $D=\Diag(d_1,d_4-d_1,d_2,d_4-d_2,d_3,d_4-d_3,d_4)$ with $d_1,d_2,d_3,d_4 \in \RR$. Analogously to the case of dimension $5$, as $\tr(D)=4d_4$, if we consider $\Cone_4(\mathfrak{h}_7)$ we obtain $d_4=1$.

  Using Proposition~\ref{prop heis}, we have that $\Cone_4(\mathfrak{h}_7)$ is given by 26 inequalities, which are given below.
  \begin{equation*}
             \begin{aligned}
              1+d_1 & >0,\\
              2-d_1 & >0,\\
              1+d_1+d_2& >0,  \\
              2+d_1-d_2& >0,\\
              2-d_1+d_2& >0,  \\
              3-d_1-d_2& >0,\\
            \end{aligned}
              \qquad \qquad
              \begin{aligned}
              1+d_2 & >0,\\
              2-d_2 & >0,\\
              1+d_2+d_3& >0,  \\
              2+d_2-d_3& >0,\\
              2-d_2+d_3& >0,  \\
              3-d_2-d_3& >0,\\
              \end{aligned}
            \qquad \qquad
              \begin{aligned}
               1+d_3 & >0,\\
              2-d_3 & >0,\\
              1+d_1+d_3& >0,  \\
              2+d_1-d_3& >0,\\
              2-d_1+d_3& >0,  \\
              3-d_1-d_3& >0,\\
            \end{aligned}
    \end{equation*}
\newline
    \begin{equation*}
             \begin{aligned}
              1+d_1+d_2+d_3& >0,  \\
              2-d_1+d_2+d_3& >0,\\
              2+d_1-d_2+d_3& >0,  \\
              2+d_1+d_2-d_3& >0,\\
            \end{aligned}
              \qquad \qquad
              \begin{aligned}
              3-d_1-d_2+d_3& >0,  \\
              3-d_1+d_2-d_3& >0,\\
              3+d_1-d_2-d_3& >0,  \\
              4-d_1-d_2-d_3& >0.\\
            \end{aligned}
    \end{equation*}
These equations define a $26$-sided polygon with $24$ vertices , see Figure~\ref{figh7}.
\begin{paracol}{2}
\begin{figure}[h!]
\begin{tikzpicture}[scale=0.8]
         \draw[very thin,color=gray,step=1cm] (-2.2,-2.2) grid (4.2,3.2);
         \foreach \x in {-2,-1,0,1,2,3,4}
         \draw (\x cm,1pt) -- (\x cm,-1pt) node[anchor=north] {$\x$};
         \foreach \y in {-2,-1,1,2}
         \draw (1pt,\y cm) -- (-1pt,\y cm) node[anchor=east] {$\y$};

         \draw[->] (-2,0) -- (3.8,0) node[right] {$d_1$};
         \draw[->] (0,-2) -- (0,3) node[above] {$d_2$};

         \draw[very thick, red, fill, opacity=0.2]   (2,1)-- (2,0)--(1,-1)--(0,-1)--(-1,0)--(-1,1)--(0,2)--(1,2)--(2,1);
         \draw[dashed,very thick, red]  (2,1)-- (2,0)--(1,-1)--(0,-1)--(-1,0)--(-1,1)--(0,2)--(1,2)--(2,1);

         \draw[color=red, fill=red] (2,1) circle (1.5pt) node[right] {$(2,1)$};
         \draw[color=red, fill=red] (2,0) circle (1.5pt) node[right] {$(2,0)$};
         \draw[color=red, fill=red] (1,-1) circle (1.5pt) node[right] {$(1,-1)$};
         \draw[color=red, fill=red] (0,-1) circle (1.5pt) node[below] {$(0,-1)$};
         \draw[color=red, fill=red] (-1,0) circle (1.5pt) node[left] {$(-1,0)$};
         \draw[color=red, fill=red] (-1,1) circle (1.5pt) node[left] {$(-1,1)$};
         \draw[color=red, fill=red] (0,2) circle (1.5pt) node[above] {$(0,2)$};
         \draw[color=red, fill=red] (1,2) circle (1.5pt) node[right] {$(1,2)$};

         \draw[color=red, fill=red] (0.5,0.5) circle (1.5pt);
\end{tikzpicture}
         \caption{$\Cone_3(\mathfrak{h}_5)$}
         \label{figh5}
\end{figure}
\switchcolumn
\begin{figure}[h!]
        \includegraphics[width=4.8cm, height=4.8cm]{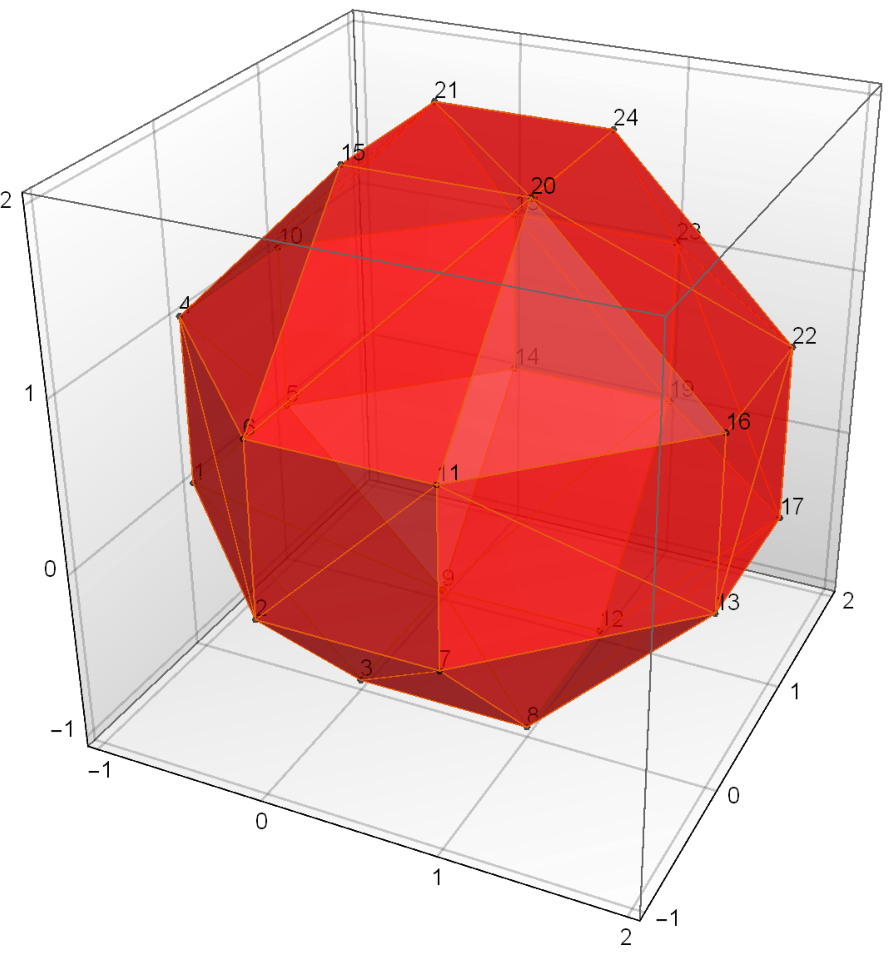}
\caption {$\Cone_4(\mathfrak{h}_7)$}
\label{figh7}
\end{figure}
\end{paracol}

Results of \cite{NklNkn} enables us to say that Conjecture 1 holds for Heisenberg Lie algebra of dimension $2n+1$.  Let $\ggo$ be a
solvable Lie algebra with the nilradical $\mathfrak{h}_{2n+1}$.  For any $X\in\ggo$, the vector $e_{2n+1}$ is an eigenvector of the restriction of $\ad_X$ to $\mathfrak{h}_{2n+1}$, so that $[X,e_{2n+1}]=\lambda(X)e_{2n+1}$ for a one-form $\lambda$ on $\ggo$, and moreover, $(\ad_X)_{|\mathfrak{h}_{2n+1}}$  descends to a well-deﬁned linear map $\Phi(X)\in\End( \mathfrak{h}_{2n+1}/\zg)$. Let $d_i(X)\in\CC, \ i=1,\ldots,2n$ be the eigenvalues of $\Phi(X)$, each listed with its algebraic multiplicity.  In \cite[Theorem 3]{NklNkn} it is proved that the cone $\tgo(\mathfrak{h}_{2n+1})_{srn}$ is given by
$$\lambda(Y)+\sum_{i: \ Re\ d_i(Y)<0}Re\ d_i >0.$$
In our notation, if a diagonal derivation of $\mathfrak{h}_{2n+1}$ is given by $D=\Diag(d_1,d_{n+1}-d_1,d_2,d_{n+1}-d_2,\ldots,d_n,d_{n+1}-d_n,d_{n+1})$, this means that $$d_{n+1}+\sum_{i:\ d_i<0}d_i+\sum_{j:\ d_{n+1}<d_j}(d_{n+1}-d_j) >0$$ are the equations defining $\tgo(\mathfrak{h}_{2n+1})_{srn}$, we claim that these are the same equations which define $\Cone(\mathfrak{h}_{2n+1})$.
\begin{theorem}
  $\Cone(\mathfrak{h}_{2n+1})=\tgo(\mathfrak{h}_{2n+1})_{srn}, \text{ for all } n\in\NN$
\end{theorem}
\begin{proof}
Given $D=\Diag(d_1,d_{n+1}-d_1,d_2,d_{n+1}-d_2,\ldots,d_n,d_{n+1}-d_n,d_{n+1})\in\tgo(\mathfrak{h}_{2n+1})$, if $D$ is strongly Ricci negative then $d_{n+1}>0$ by the necessary condition of \cite[Theorem 2]{NklNkn}. Then, if $D\in\tgo(\mathfrak{h}_{2n+1})_{srn}$ each $d_i\in\RR$ has three options, being less than $0$, between (or equal) $0$ and $d_{n+1}$ or greater than $d_{n+1}$, and as the cone is $W_{ort}$-invariant we can assume that from $i=1\ldots, r$, $d_i<0$, $0\leq d_i\leq d_{n+1}$ if $i=r+1,\ldots, s,$ and if $i=s+1,\ldots, n$ then $d_i>d_{n+1}$. In this case, following \cite[Theorem 3]{NklNkn}, we have
$$d_{n+1}+d_1+\ldots+d_r+(n-s)d_{n+1}-d_{s+1}-\ldots-d_n>0.$$
And if we take $k=r+n-s$ and $l=n-s$ we have the same equation in $\Cone(\mathfrak{h}_{2n+1})$ by \eqref{conoheis}.
\end{proof}

  \section{Filiform Lie algebra $L_n$}\label{fil}
  We now calculate $\Cone(L_n)$ for all $n \in \NN$, and we compare these results with \cite{NklNkn}.

\begin{proposition}
  Let $L_n$ be the filiform Lie algebra defined by the basis $\{e_1,e_2,\ldots, e_{n}\}$ and the Lie brackets:
  $$ \mu(e_1,e_2)=e_{3}, \quad \mu(e_1,e_3)=e_{4}, \quad \dots, \quad \mu(e_{1},e_{n-1})=e_{n}.$$ For $ D:= \Diag(d_1,d_2,d_1+d_2,2d_1+d_2,\ldots,(n-2)d_1+d_2) \in \tgo(L_n)$, the cone $\Cone(L_n)$ is defined by the following equations,
  $$(n-2)d_{1}+d_2 >0 \qquad \text{and} \qquad  \frac{(n-1)(n-2)}{2}d_1+(n-1)d_2 >0. $$
\end{proposition}
\begin{proof}
  Let $L_{n}$ be the  filiform Lie algebra and let $ D= \Diag(d_1,d_2,d_1+d_2,2d_1+d_2,\ldots,(n-2)d_1+d_2) \in \tgo(L_n).$
 Note that $ \tr(D)=(\sum_{i=1}^{n-2}i+1)d_1+(n-1)d_2$ and $D|_{\zg(L_{n})}=(n-2)d_1+d_2$.

 $D$ is generic if and only if:
  $$d_1,d_2 \neq 0, \quad d_1 \neq d_{2}, \quad d_2 \neq -jd_1 \quad \text{for all} \ j=1 \ldots n-3. $$

 In this case $D$ has pairwise different eigenvalues, this implies that $\tgo(L_{n})$ is multiplicity free, $G_{\tgo(L_{n})}=\Diag(L_n)_{>0}$ and $ \overline{G_{\tgo(L_{n})}\cdot \mu}$ is given by:
  $$ \lambda(e_1,e_i)=\alpha_i e_{i+1} \quad \text{for all} \ i=2 \ldots n-1,$$
  where $\alpha_i \geq 0. $
  From~\eqref{defmm2}, we have that
   {\small
          \begin{align*}
          \mm(\lambda)&=\frac{2}{|\lambda |^2}\Diag(-\alp_2-\ldots-\alp_{n-1},-\alp_2,\alp_2-\alp_3,\alp_3-\alp_4,\ldots,\alp_{n-2}-\alp_{n-1},\alp_{n-1})\\
          &= \CH(F_{1i}^{i+1} \quad  \ i=2 \ldots n-1).
          \end{align*}
          }%
    By \eqref{def cono}, a derivation $D$ is in $\Cone(L_{n})$ if and only if there exist $a_2,a_3,\ldots, a_{n-1} >0$ and $ E \in \Diag(L_{n})_{>0}$ such that
   $$ D=\sum_{i=2}^{n-1}a_iF_{1i}^{i+1} + E.$$
   That is,
   \begin{equation*}
               \left\{
             \begin{aligned}
              d_1+a_2+\ldots+a_{n-1}& >0,  \\
              d_{2}+a_2& >0,\\
              (d_1+d_2)-a_2+a_3& >0,  \\
              \end{aligned}
              \right.
              \qquad\qquad
              \left\{
              \begin{aligned}
              (2d_1+d_2)-a_3+a_4& >0,\\
              \vdots \\
              (n-3)d_1+d_2-a_{n-2}+a_{n-1}& >0,  \\
              (n-2)d_1+d_2-a_{n-1}& >0.\\
            \end{aligned}
            \right.
   \end{equation*}
   This holds if and only if there exists $a_{n-1}>0$ such that
  \begin{align*}
    (n-2)d_{1}+d_2 > a_{n-1}> & -(n-3)d_1-d_2+a_{n-2}, \\
    &  -d_1-a_2-\ldots-a_{n-2},
  \end{align*}
  if and only if there exists $a_{n-2}>0$ such that
   \begin{align*}
    [(n-2)+(n-3)]d_1+2d_2> a_{n-2} > & -(n-4)d_1-d_2+a_{n-3}, \\
    &  [-1-(n-2)]d_1-d_2-a_2-\ldots-a_{n-3}.
  \end{align*}
  In the same manner we can see that there exists $a_{n-3}>0$ such that
   \begin{align*}
    [(n-2)+(n-&3)+(n-4)]d_1+3d_2 > a_{n-3} >  -(n-5)d_1-d_2+a_{n-4}, \\
    &  [-1-2(n-2)-(n-3)]d_1-(1+2)d_2-a_2-\ldots-a_{n-4}.
  \end{align*}
 If we apply this argument over and over again, we obtain that there exists $a_3>0$ such that
   $$[(n-2)+(n-3)+(n-4)+\ldots+3+2]d_1+(n-3)d_2> a_{3} >  -d_1-d_2+a_{2},$$
    $$ [-1-(n-4)(n-2)-(n-5)(n-3)-(n-6)(n-4)-\ldots-3]d_1-(1+2+3+\ldots+(n-4))d_2-a_2,$$
   and there exists $a_2>0$ such that
   $$ \left[\sum_{i=2}^{n-1}(n-i)\right]d_{1}+(n-2)d_2 > a_{2}>  -d_2, \ \left[-1-\sum_{i=2}^{n-2}(i-1)i \right]d_1-\left(\sum_{i=1}^{n-3}i \right)d_2. $$
  Hence,
  $$\left[\sum_{i=2}^{n-1}(n-i)\right]d_{1}+(n-1)d_2>0 \qquad \text{and} $$ $$\left[1+\sum_{i=2}^{n-2}(i-1)i+\sum_{i=2}^{n-1}(n-i)\right]d_1+\left[\frac{(n-3)(n-2)}{2}+(n-2)\right]d_2>0$$
  Since,
  $$\sum_{i=2}^{n-1}(n-i)=\tfrac{(n-1)(n-2)}{2} \quad \text{and} \quad    \sum_{i=2}^{n-2}(i-1)i=\tfrac{n}{6}(2n^2-9n+13)-1-\tfrac{(n-2)(n-1)}{2},$$
 the previous system is equivalent to the next one:
 \small{
    \begin{equation*}
               \left\{
             \begin{aligned}
              (n-2)d_{1}+d_2& >0  \\
             [(n-2)+(n-3)]d_1+2d_2& >0\\
              [(n-2)+(n-3)+(n-4)]d_1+3d_2& >0  \\
              \vdots \\
              [(n-2)+(n-3)+\ldots+2]d_1+(n-3)d_2& >0  \\
              \end{aligned}
             \right.
             \qquad \qquad
             \left\{
             \begin{aligned}
              \left[\sum_{i=2}^{n-1}(n-i)\right]d_{1}+(n-2)d_2& >0\\
              \tfrac{(n-1)(n-2)}{2}d_1+(n-1)d_2& >0 \\
              \tfrac{n}{6}(2n^2-9n+13)d_1+\tfrac{(n-1)(n-2)}{2}d_2 & >0 \\
            \end{aligned}
            \right.
   \end{equation*}
}
We claim that
$$n-2 \geq\tfrac{\sum_{i=2}^{j}(n-i)}{j-1}\geq \tfrac{n-2}{2} \quad \text{for all} \ n>j$$
$$ n-2\geq \tfrac{n(2n^2-9n+13)}{3(n-1)(n-2)} \geq \tfrac{n-2}{2} \quad \text{for all} \ n>3.$$

Indeed,
\begin{equation*}
  \tfrac{\sum_{i=2}^{j}(n-i)}{j-1}= n - \tfrac{j^2+j}{2(j-1)}+\tfrac{1}{j-1} \geq \tfrac{n-2}{2} \ \text{if and only if} \ n \geq j, \ \text{and}
\end{equation*}
\begin{equation*}
  n-2\geq n - \tfrac{j^2+j-2}{2(j-1)} \ \text{if and only if} \ 2  \leq \tfrac{j+2}{2},\ \text{i.e. it holds for all} \  j\geq2.
\end{equation*}
The other inequalities are equivalent to prove:
$$ n^3 -6n^2 +11n-12 \geq0 \ \text{and} \ \tfrac{1}{2}n^3+\frac{3}{2}n^2-n-6 \geq 0  \quad \text{for all} \ n>3.$$
If we consider those as real cubic functions, $f(4)\geq0$ in both cases and the functions are increasing from $x=4$. So, the proof is complete (see Figure~\ref{conefil}).
\end{proof}

According to \cite{NklNkn} these equations define $\tgo(L_n)_{srn}$, therefore Conjecture 1 is valid in this case.
\begin{corollary}
  $\Cone(L_n)=\tgo(L_n)_{srn}$, for all $n>3$.
\end{corollary}

\begin{figure}[h!]
\begin{tikzpicture}[scale=0.8]
\draw[very thin,color=gray,step=1cm] (-2.2,-4.5) grid (4.5,5.5);
\foreach \x in {-1,0,1,2,3,4,5}
   \draw (\x cm,1pt) -- (\x cm,-1pt) node[anchor=north] {$\x$};
\foreach \y in {0,1,2,3,4,5}
   \draw (1pt,\y cm) -- (-1pt,\y cm) node[anchor=east] {$\y$};

\draw[->] (-3,0) -- (6,0) node[right] {$d_1$};
\draw[->] (0,-2) -- (0,5) node[above] {$d_2$};

\draw[dashed, very thick, blue]  (0,0) -- (-1,5);
\draw[dashed, very thick, blue]  (0,0) -- (2,-3.75);

\draw[orange, fill, opacity=0.2]  (-1,5) .. controls (4,6) .. (2,-3.75) -- (0,0) -- (-1,5);

\draw (3.8,1.5) node[color=orange, thick] {$\Cone(L_n)$};

\draw (2.5,-4) node[color=blue,very thick] {$d_2=-\frac{n-2}{2}d_1$};
\draw (-2.2,5.3) node[color=blue,very thick] {$d_2=-(n-2)d_1$};

\end{tikzpicture}
\caption{$\Cone(L_n)$}
\label{conefil}
\end{figure}
\newpage

\textit{Acknowledgements.}  I would like to thank to my Ph.D. advisor Dr. Jorge Lauret for his continued guidance during the preparation of this paper.

\end{document}